\documentclass{amsart}

\usepackage{amsthm,amsmath,amsfonts,amssymb,array,enumerate,verbatim}
\usepackage[all]{xy}
\usepackage{url}

\newcommand{\isomto}{\overset{\sim}{\rightarrow}}

\newcommand{\CC}{\mathbb{C}}
\newcommand{\ZZ}{\mathbb{Z}}

\newcommand{\QQ}{\mathbb{Q}}
\newcommand{\cO}{\mathcal{O}}

\newcommand{\fp}{\mathfrak{p}}
\newcommand{\fq}{\mathfrak{q}}

\newcommand{\fM}{\mathfrak{M}}
\newcommand{\fW}{\mathfrak{W}}
\newcommand{\fS}{\mathfrak{S}}
\newcommand{\cW}{\mathcal{W}}

\newcommand{\cK}{\mathcal{K}}
\newcommand{\cS}{\mathcal{S}}

\newcommand{\cZ}{\mathcal{Z}}

\newcommand{\cV}{\mathcal{V}}
\newcommand{\cH}{\mathcal{H}}
\DeclareMathOperator{\Hom}{Hom}

\DeclareMathOperator{\Ind}{Ind}
\DeclareMathOperator{\cInd}{c-Ind}
\DeclareMathOperator{\End}{End}
\DeclareMathOperator{\Rep}{Rep}

\DeclareMathOperator{\Frac}{Frac}
\DeclareMathOperator{\Spec}{Spec}

\DeclareMathOperator{\cosoc}{cosoc}

\DeclareMathOperator{\Img}{Im}

\newtheorem{theorem}{Theorem}[section]
\newtheorem{lemma}[theorem]{Lemma}
\newtheorem{proposition}[theorem]{Proposition}
\newtheorem{corollary}[theorem]{Corollary}
\newtheorem{definition}[theorem]{Definition}

\begin{document}

\title{Gamma factors of pairs and a local converse theorem in families}
\author{Gilbert Moss}

\begin{abstract}
We prove a $GL(n)\times GL(n-1)$ local converse theorem for $\ell$-adic families of smooth representations of $GL_n(F)$ where $F$ is a finite extension of $\QQ_p$ and $\ell\neq p$. Along the way, we extend the theory of Rankin-Selberg integrals, first introduced in \cite{jps2}, to the setting of families, continuing previous work of the author \cite{moss1}.
\end{abstract}

\maketitle

\section{Introduction}
Let $F$ be a finite extension of $\QQ_p$. A local converse theorem is a result along the following lines: given $V_1$ and $V_2$ representations of $GL_n(F)$, if $\gamma(V_1\times V',X,\psi)=\gamma(V_2\times V',X,\psi)$ for all representations $V'$ of $GL_{n-1}(F)$, then $V_1$ and $V_2$ are the same. There exists such a converse theorem for complex representations: if $V_1$, $V_2$, and $V'$ are irreducible admissible generic representations of $GL_n(F)$ over $\CC$, ``the same'' means isomorphic (\cite{hen_converse}). It is a conjecture of Jacquet that it should suffice to let $V'$ vary over representations of $GL_{\lfloor\frac{n}{2}\rfloor}(F)$, or in other words a $GL(n)\times GL(\lfloor\frac{n}{2}\rfloor)$ converse theorem should hold. In this paper we construct $\gamma(V\times V',X,\psi)$ and prove a $GL(n)\times GL(n-1)$ local converse theorem in the setting of $\ell$-adic families. We deal with admissible generic families that are not typically irreducible, so ``the same'' will mean that $V_1$ and $V_2$ have the same supercuspidal support. Over families, there arises a new dimension to the local converse problem: determining the smallest coefficient ring over which the twisting representations $V'$ can be taken while still having the theorem hold. Before stating the result, we develop some notation.

A family of $GL_n(F)$-representations is an $A[GL_n(F)]$-module $V$ where $A$ is a Noetherian ring in which $p$ is invertible. The development of the theory is facilitated if $A$ is also a $W(k)$-algebra, where $k$ is an algebraically closed field of characteristic $\ell$, with $\ell\neq p$, and $W(k)$ denotes the Witt vectors (recall that $W(\overline{\mathbb{F}_{\ell}})$ is isomorphic to the $\ell$-adic completion of the ring of integers in the maximal unramified extension of $\QQ_{\ell}$); this is also the setting of Galois deformations. Given $\fp$ in $\Spec(A)$ with residue field $\kappa(\fp):=A_{\fp}/\fp A_{\fp}$, the fiber $V\otimes_A\kappa(\fp)$ gives a representation on a $\kappa(\fp)$-vector space. 

In this paper we consider admissible generic $A[GL_n(F)]$-modules which are co-Whittaker (Definition \ref{defnofcowhitt}). Each fiber of a co-Whittaker family admits a unique surjection onto an absolutely irreducible space of Whittaker functions. Emerton and Helm conjecture the existence of a map from the set of continuous Galois deformations over $W(k)$-algebras $A$ to the set of co-Whittaker $A[G]$-modules (in the setting where $A$ is complete, local with residue field $k$, reduced, and $\ell$-torsion free). Their definition is motivated by global constructions: the smooth dual of the $\ell\neq p$ tensor factor of Emerton's $\ell$-adically completed cohomology \cite{em_lg} is an example of a co-Whittaker module.

The local Langlands correspondence in families is uniquely characterized by requiring that it interpolate (a dualized generic version of) classical local Langlands in characteristic zero \cite[Thm 6.2.1]{eh}. The possibility of characterizing this correspondence using local constants forms one of the initial motivations for this work.

Co-Whittaker modules admit essentially one map into the space $\Ind_N^G\psi$, where $N$ is the subgroup of unipotent upper-triangular matrices, and $\psi$ is a fixed generic character. The image of this map is called the Whittaker space of $V$, denoted $\cW(V,\psi)$. Given $W$ in $\cW(V,\psi)$ and $W'$ in $\cW(V',\psi)$, we define the local Rankin-Selberg formal series $\Psi(W,W',X)$. In the setting of families there is no need to restrict ourselves to the situation where $V$ and $V'$ have the same coefficient ring. Therefore, the base ring is taken to be $R:=A\otimes_{W(k)}B$ where $A$ and $B$ are Noetherian $W(k)$-algebras, $V$ is an $A[GL_n(F)]$-module, and $V'$ is a $B[GL_m(F)]$-module.

Classically the local integrals form elements of $\CC(q^{-s})$ where $q$ is the order of the residue field of $F$. In this paper we replace the complex variable $q^{-s+\frac{n-m}{2}}$ with a formal variable $X$ and use purely algebraic methods. Our coefficient ring $R$ may not be a domain, so the analogue of $\CC(q^{-s})$ is more subtle. As in \cite{moss1}, the formal series $\Psi(W,W',X)$ will define an element of the fraction ring $S^{-1}(R[X,X^{-1}])$ where $S$ is the multiplicative subset of $R[X,X^{-1}]$ consisting of polynomials whose first and last coefficients are units; this is proved in \S \ref{rationalitysection}. This ring enables us to compare the objects on either side of a functional equation, which is proved in \S \ref{functionalequationsection}. The proofs of rationality and the functional equation follow the same overall pattern as the results for the $GL(n)\times GL(1)$ case, which is the subject of \cite{moss1}. In the functional equation for $\Psi(W,W',X)$, there is a term which remains constant as $W$ and $W'$ vary; this is the gamma factor $\gamma(V\times V',X,\psi)$.

In \S \ref{notationconventions} we give a definition of supercuspidal support for co-Whittaker families, and show that two co-Whittaker families have the same supercuspidal support if and only if they have the same Whittaker space. Our main theorem then states that gamma factors uniquely determine supercuspidal support:

\begin{theorem}
\label{intromainthm}
Let $A$ be a finite-type $W(k)$-algebra which is reduced and $\ell$-torsion free, and let $\cK:=\Frac(W(k))$. Suppose $V_1$ and $V_2$ are two co-Whittaker $A[GL_n(F)]$-modules. There is a finite extension $\cK'$ of $\cK$ such that, if $\gamma(V_1\times V',X,\psi) = \gamma(V_2\times V',X,\psi)$ for all absolutely irreducible generic integral representations $V'$ of $GL_{n-1}(F)$ over $\cK'$, then $V_1$ and $V_2$ have the same supercuspidal support (equivalently, $\cW(V_1,\psi)=\cW(V_2,\psi)$).
\end{theorem}
Recall that a $\cK'[G]$-module $V'$ is integral if it admits a $G$-stable $\cO_{\cK'}$-lattice, where $\cO_{\cK'}$ is the ring of integers of $\cK'$.

Thus, in the reduced and $\ell$-torsion free setting, our converse theorem shows it suffices to take the coefficient ring of the twisting representations $V'$ to be no larger than the ring of integers in a finite extension of $\cK$. The hypothesis that $A$ is a finite-type, reduced, and $\ell$-torsion free enables us to implement a key vanishing lemma, Theorem \ref{rikka2}, which is explained below. If $V_1$ and $V_2$ live within a single block of the category $\Rep_{W(k)}(GL_n(F))$, the finite extension $\cK'$ appearing in our converse theorem depends only on this block. Finding the smallest possible extension $\cK'$ for each block will be the subject of future investigation.

If $E$ is a finite extension of $\cK$, and $V$ is an absolutely irreducible generic integral representation of $GL_n(F)$ over $E$, then in particular it has a sublattice $L$ which is co-Whittaker \cite[3.3.2 Prop]{eh}, \cite[I.9.7]{vig}, and the supercuspidal support of $L$ determines $V$ up to isomorphism. Thus our converse theorem gives as a special case the following integral converse theorem:
\begin{corollary}
Let $V_1$, $V_2$ be two absolutely irreducible generic integral representations of $GL_n(F)$ over $E$. There is a finite extension $\cK'/\cK$ such that if $\gamma(V_1\times V',X,\psi)=\gamma(V_2\times V',X,\psi)$ for all absolutely irreducible generic integral representations $V'$ of $GL_{n-1}(F)$ over $\cK'$, then $V_1\cong V_2$.
\end{corollary} 

In Section \ref{conversetheoremsection} we prove Theorem \ref{intromainthm} by employing the functional equation, following the method of \cite{hen_converse} and \cite[Thm 7.5.3]{jps1}.

In proving the converse theorem, there is a key vanishing lemma (Theorem \ref{rikka2} in this paper) that is well-known in the setting of complex representations, but is more difficult in families. It says that given any smooth compactly supported function $H$ on $GL_n(F)$ with $H(ng)=\psi(n)H(g)$, $n\in N$, the vanishing of $H$ can be detected by the convolutions of $H$ with the Whittaker functions of a sufficiently large collection of representations. This result was originally proven over $\CC$ in \cite[Lemme 3.5]{jps4} by using harmonic analysis to decompose a representation as the direct integral of irreducible representations. A purely algebraic analogue of this decomposition was obtained in \cite{bh_whitt} by viewing the representation as a sheaf on the spectrum of the Bernstein center. As an application of these algebraic techniques, a new proof of this vanishing lemma (over $\CC$) is given over in \cite{bh_whitt}. It has been observed in the $\ell$-modular setting in \cite{vig_bern, vig_whitt}, and more recently in the integral setting in \cite{h_bern, h_whitt} that Bernstein's algebraic approach to Fourier theory and Whittaker models applies to representations over coefficient rings other than $\CC$. In Section \ref{proofofvanishingthm} we apply the theory of the integral Bernstein center, developed in \cite{h_bern, h_whitt}, to prove the key vanishing lemma (and thus the converse theorem) in the case when $A$ is a flat finite-type reduced $W(k)$-algebra. The geometric methods in Section \ref{proofofvanishingthm}  require $A$ to be reduced in order that a certain subset $D$ of desirable points is open (Lemmas \ref{torsionfree}, \ref{Bexists}); $A$ is required to be finite and flat over $W(k)$ so that a certain structure map is open, and the image of $D$ in $\Spec(\cZ)$ intersects a dense set.

Converse theorems in the complex setting have a long history dating back to Hecke, and for $GL(n)$ in the local setting have been studied in \cite{jacquet_langlands,jps1,jps2,hen_converse,cogshap_converse_II,chen, JNS}, among others.

The possibility of characterizing the mod-$\ell$ local Langlands correspondence via a converse theorem for $\ell$-modular local constants has been investigated in \cite{vig_epsilon} for the supercuspidal case when $n=2$. The Rankin-Selberg convolutions in this paper expand on recent results on Rankin-Selberg convolutions in the $\ell$-modular setting in \cite{km}. However, in $\ell$-adic families, the analogue of the $L$-factor does not seem to behave well \cite[\S 0]{moss1}, which is why we focus at present only on the local integrals $\Psi(W,W',X)$ and the gamma factor. In the $\ell$-modular setting, where $A=k$, it appears that the approach of \cite{bh_whitt} to the key vanishing lemma would require techniques able to handle $\ell$-torsion, going beyond those presented here.

We assume $F$ has characteristic zero because at present this is the only setting in which the theory of the integral Bernstein center, which we rely on, exists at present.

The methods of this paper can be adapted to show that the coefficients of the universal Rankin-Selberg gamma factor provide a set of generators for the integral Bernstein center. A proof of this will appear in forthcoming work.

\subsection{Acknowledgements}
The author is very grateful to his advisor David Helm for his guidance and support. He would like to thank Keenan Kidwell and Cory Colbert for many helpful conversations on commutative algebra, Travis Schedler for several helpful comments, and James Cogdell, Guy Henniart, Robert Kurinczuk, and Nadir Matringe for their continued interest in these results.

\section{Notation and Definitions}
\label{notationconventions}
Let $F$ be a finite extension of $\QQ_p$, let $q$ be the order of its residue field, and let $k$ be an algebraically closed field of characteristic $\ell$, where $\ell\neq p$ is an odd prime.  $\cO_F$ will denote the ring of integers in $F$, $U_F$ will denote $\cO_F^{\times}$, and $\varpi$ will denote a uniformizer. The letter $G$ or $G_n$ will always denote the group $GL_n(F)$.  We will denote by $W(k)$ the ring of Witt vectors over $k$. The assumption that $\ell$ is odd is made so that $W(k)$ contains a square root of $q$. When $\ell=2$ all the arguments presented will remain valid, after possibly adjoining a square root of $q$ to $W(k)$. $A$ is always a Noetherian commutative ring which is a $W(k)$-algebra, with additional ring theoretic conditions in various sections of the paper, and for a prime $\fp$ we denote by $\kappa(\fp)$ the residue field $A_{\fp}/\fp A_{\fp}$.  

For a group $H$, we denote by $\Rep_A(H)$ the category of smooth representations of $H$ over the ring $A$, i.e. $A[H]$-modules for which every element is stabilized by an open subgroup of $H$. We will sometimes drop the subscript and write $\Rep(H)$ to mean $\Rep_A(H)$, and even when this category is not mentioned, all representations are presumed to be smooth. An $A[H]$-module is admissible if for every compact open subgroup $U$, the set of $U$-fixed vectors is finitely generated as an $A$-module. 

If $V$ is a smooth representation of $H$ over a ring $A$, and $\theta:H\rightarrow A^{\times}$ is a smooth character, we denote by $V_{H,\theta}$ the quotient $V/V(H,\theta)$ where $V(H,\theta)$ is the sub-$A$-module generated by elements of the form $hv-\theta(h)v$ for $h\in H$ and $v\in V$. Given a representation $\sigma$ of a closed subgroup $K$ we define the induction $\Ind_K^H\sigma$ to be the $H$-module (by right translation) of functions $f:H\rightarrow \sigma$ satisfying $f(kh)=\sigma(k)f(h)$, $k\in K$, $h\in H$, and which are invariant under right translation by a compact open subgroup of $H$. The module $\cInd_K^H\sigma$ consists of those function in $\Ind_K^H\sigma$ which are compactly supported modulo $K$.

Given a standard parabolic subgroup $P$ of $GL_n(F)$ (i.e. a subgroup consisting of block upper triangular matrices), it has a unipotent radical $M$ (of strictly block upper triangular matrices) such that $P=LM$ for a subgroup $L$, a standard Levi subgroup, of block diagonal matrices. The functor $V\mapsto V_{L,\textbf{1}}$ (after restricting to $P$) is called the Jacquet functor associated to $L$, and we denote this functor by $J_L$. An $A[G]$-module $V$ is called \textsl{cuspidal} if $J_LV=0$ for all Levi subgroups $L\neq G$. $J_L$ has a right adjoint, given by parabolic induction, which takes a representation $V\in \Rep_A(L)$, inflates it to a representation of $P=LM$ by letting $M$ act trivially, and then taking the compactly induced representation $\cInd_P^GV$. This functor is denoted $i_P^G$. Both $i_P^G$ and $J_L$ are defined when $G$ is an arbitrary reductive group.

Given $V\in \Rep_A(G)$, this adjunction implies \cite[II.2.3]{vig} that $V$ is cuspidal if and only if all $G$-homomorphisms from $V$ to a parabolic induction $i_P^GW$ are zero for all $W\in \Rep(L)$, $L\neq G$. If $A$ is a field, then a simple $A[G]$-module is called \textsl{supercuspidal} if it is not isomorphic to a subquotient of $i_P^GW$ for any $W\in \Rep(L)$, $L\neq G$. If $A$ is a field of characteristic zero, cuspidal representations are supercuspidal.

\begin{definition}
$V$ in $\Rep_A(G)$ will be called $G$-finite if it is finitely generated as an $A[G]$-module.
\end{definition}

We denote by $N_n$, or just $N$, the subgroup of $G_n$ consisting of all unipotent upper-triangular matrices. Let $\psi: F\rightarrow W(k)^{\times}$ be a nontrivial additive character with open kernel. For a $W(k)$-algebra $A$, $\psi_A$ will denote $\psi\otimes_{W(k)}A$. $\psi$ defines a character on any subgroup of $N_n$ by $(u)_{i,j}\mapsto \psi(u_{1,2}+\dots+u_{n-1,n})$. We will abusively denote this character by $\psi$ as well.
\begin{definition}
For $V$ in $\Rep_A(G_n)$, we say that $V$ is of Whittaker type if $V_{N,\psi}$ is free of rank one as an $A$-module. We say $V$ is generic if $V_{N,\psi}$ is nonzero.
\end{definition}

If $V$ is of Whittaker type, $\Hom_A(V/V(N_n,\psi),A) = \Hom_{N_n}(V,\psi)$ is free of rank one, so we may choose a generator $\lambda$ in $\Hom_{N_n}(V,\psi)$.  For any $v$ in $V$, define $W_v\in \Ind_{N_n}^{G_n}\psi$ as $W_v: g\mapsto \lambda(gv)$.  This is called a Whittaker function and has the property that $W(nx) = \psi(n)W(x)$ for $n\in N_n$.  $v\mapsto W_v$ defines a $G_n$-equivariant homomorphism $V\rightarrow \Ind_{N_n}^{G_n}\psi$.  The image is an $A[G]$-module independent of the choice of $\lambda$.  The map $v\mapsto W_v$ is precisely the generator of $\Hom_{G_n}(V,\Ind_{N_n}^{G_n}\psi)$ corresponding to the generator $\lambda$ of $\Hom_{N_n}(V,\psi)$ under Frobenius reciprocity. The image of the homomorphism $v\mapsto W_v:V\rightarrow \Ind_{N_n}^{G_n}\psi$ is called the space of Whittaker functions of $V$ and is denoted $\cW(V,\psi)$ or just $\cW$.

Note that the map $V\rightarrow \cW(V,\psi)$ is surjective but not necessarily an isomorphism, and different $A[G]$-modules of Whittaker type can have the same space of Whittaker functions \cite[see][Lemma 1.6]{moss1}.

For each $m\leq n$, we let $G_m$ denote $GL_m(F)$ and embed it in $G$ via $(\begin{smallmatrix}G_m & 0 \\ 0 & I_{n-m}\end{smallmatrix})$.  We let $\{1\} = P_1 \subset \cdots \subset P_n$ denote the mirabolic subgroups of $G_1\subset \cdots \subset G_n$, which are given $P_m:=\left\{(\begin{smallmatrix} g_{m-1} & x \\ 0 & 1\end{smallmatrix}):g_{m-1}\in G_{m-1},\ x\in F^{m-1}\right\}$.  We also have the unipotent upper triangular subgroup $U_m := \left\{(\begin{smallmatrix} I_{m-1} & x \\ 0 & 1\end{smallmatrix}):x\in F^{m-1}\right\}$ of $P_m$ such that $P_m = U_mG_{m-1}$. We will sometimes write $P$ for $P_n$.

We can repeat the Whittaker functions construction for the restriction to $P_n$ of representations $V$ in $\Rep(G_n)$ of Whittaker type.  In particular, by restricting the argument of the Whittaker functions $W_v$ to elements of $P_n$, we get a $P_n$-equivariant homomorphism $V\rightarrow \Ind_{N_n}^{P_n}\psi$. The image of the homomorphism $V\rightarrow \Ind_{N_n}^{P_n}\psi:v\mapsto W_v$ is called the Kirillov functions of $V$ and is denoted $\cK(V,\psi)$ or just $\cK$.  It carries a representation of $P_n$ via $pW_v = W_{pv}$. 

Following \cite{b-zI}, we define the functor $\Phi^+:\Rep_A(P_{n-1})\rightarrow \Rep(P_n)$ which sends $V$ to $\Ind_{P_{n-1}U_n}^{P_n}V$ where $U_n$ acts via $\psi$, and the functor $(-)^{(n)}$, which sends $V$ to $V_{N,\psi}$.

There is a particularly important $P_n$-representation that naturally embeds in the restriction to $P_n$ of any Whittaker type representation $V$ in $\Rep(G_n)$:
\begin{definition}
\label{schwartzdefn}
If $V$ is in $\Rep(P_n)$, the $P_n$ representation $(\Phi^+)^{n-1}V^{(n)}$ is called the Schwartz functions of $V$ and is denoted $\cS(V)$.
\end{definition}

The Bernstein center $\mathcal{Z}$ is defined as the ring of endomorphisms of the identity functor in the category $\Rep_{W(k)}(G)$. It is a commutative ring whose elements consist of collections of compatible endomorphisms, one for each object (compatible in the sense of commuting with all morphisms in the category).  Analogous to the results of \cite{bd} for smooth representations over $\CC$, \cite[Thm 10.8]{h_bern} shows that the category $\Rep_{W(k)}(G)$ has a decomposition into full subcategories known as blocks. Given a primitive idempotent $e$ of $\mathcal{Z}$, the subcategory $e\cdot \Rep_{W(k)}(G)$, consisting of representations satisfying $eV=V$, is a block, and the ring $e\mathcal{Z}$ is its center. $e\cZ$ has an interpretation as the ring of regular functions on an affine algebraic variety over $W(k)$, whose $k$-points are in bijection with the set of unramified twists of a fixed conjugacy class of supercuspidal supports in $\Rep_k(G)$. Note that, since $\cZ$ is an infinite product of the rings $e\cZ$, it is not Noetherian.

If $A$ is a $W(k)$-algebra and $V$ is an $A[G]$-module, then $V$ is also a $W(k)[G]$-module, and we use the Bernstein decomposition of $\Rep_{W(k)}(G)$ to study $V$.
\begin{definition}
An object $V$ in $\Rep_A(G)$ is called primitive if there exists a primitive idempotent $e$ in the Bernstein center $\cZ$ such that $eV = V$.
\end{definition}

It is possible to define a Haar measure on the space $ C_c^{\infty}(G,A)$ of smooth compactly supported functions $G\rightarrow A$ by choosing a filtration $\{H_i\}_{i\geq 1}$ of compact open neighborhoods of $1$ in $G$ such that $[H_1:H_i]$ is invertible in $W(k)$.  Then we can set $\mu^{\times}(H_i) = [H_1:H_i]^{-1}$.  This automatically defines integration on the characteristic functions of the $H_i$, and one checks that extending by linearity gives a well-defined Haar integral on $C_c^{\infty}(G,A)$. Normalizing the Haar measure to be compatible with decompositions of the group requires care when $A$ has characteristic $\ell$; this is dealt with in \cite[\S 2.2]{km}.

\section{Rationality of Rankin-Selberg Formal Series}
\label{rationalitysection}
Let $A$ and $B$ be Noetherian $W(k)$-algebras and let $R = A\otimes_{W(k)}B$. Let $V$ and $V'$ be $A[G_n]$- and $B[G_m]$-modules respectively, where $m<n$, and suppose both $V$ and $V'$ are of Whittaker type. For $W\in \cW(V,\psi)$ and $W'\in \cW(V', \psi)$, we define the formal series with coefficients in $R$:
$$\Psi(W,W',X):=\sum_{r\in \ZZ}\int_{N_m\backslash\{g\in G_m:v(\det g)=r\}}\left(W(\begin{smallmatrix}g&0\\0&I_{n-m}\end{smallmatrix})\otimes W'(g)\right)X^rdg$$
and for $0\leq j \leq n-m-1$, define
\begin{align*}
\Psi(W,W',X;j):=\sum_{r\in \ZZ}\int_{M_{j,m}(F)}\int_{N_m\backslash\{g\in G_m:v(\det g)=r\}}                    \left(W\left(\begin{smallmatrix}g &&\\x&I_j &\\&&I_{n-m-j} \end{smallmatrix}\right)\otimes W'(g)\right)X^rdg dx
\end{align*}
With $\Psi(W,W',X;0) = \Psi(W,W',X)$.

\begin{lemma}
\label{finitelymanynegativeterms}
The formal series $\Psi(W,W',X;j)$ has finitely many nonzero positive powers of $X^{-1}$, thus forms an element of $R[[X]][X^{-1}]$.
\end{lemma} 
\begin{proof} 
Since \cite[Lemma 4.1.5]{jps2} is valid in this context, the proof proceeds exactly as in \cite[\S 3.1]{moss1}, after applying the Iwasawa decomposition. The Iwasawa decomposition works in this setting after choosing an appropriate Haar measure, as shown in \cite[Cor 2.9]{km}.
\end{proof}

\begin{theorem}
\label{rationality2}
Suppose $A$ and $B$ are Noetherian $W(k)$-algebras, $V$ is an $A[G_n]$-module, and $V'$ is a $B[G_m]$-module, both admissible, of Whittaker type and finitely generated over $A[G_n]$ and $B[G_m]$ respectively. Define $S$ to be the multiplicative subset of $R[X,X^{-1}]$ consisting of polynomials whose first and last coefficients are units. For any $W\in \cW(V,\psi)$, $W'\in \cW(V',\psi)$, the formal series $\Psi(W,W',X;j)$ lives in $S^{-1}(R[X,X^{-1}])$.
\end{theorem} 
When $A=B$ we can take the image of the zeta integrals in the map $S_R^{-1}(R[X,X^{-1}])\rightarrow S_A^{-1}(A[X,X^{-1}])$ induced by the map $R\rightarrow A:a_1\otimes a_2\mapsto a_1a_2$ and recover the rationality result that would be desired when both representations live over the same coefficient ring.

The remainder of this section is devoted to proving Theorem \ref{rationality2}.  In the setting of representations over a field, there is a useful decomposition of Whittaker functions into ``finite'' functions, which quickly leads to a rationality result \cite{jps1,jps2}, \cite[Prop 3.3]{km}. In the setting of rings, such a structure theorem is lacking, since the theory of finite functions achieved in \cite[Ch. 1 \S 8]{jacquet_langlands} only exists at present for vector spaces over fields. It seems difficult to extend this theory to modules which are not even free. However, certain elements of the original proof, combined with a translation property of the integral (after restricting to the torus), can be used to prove rationality. 

As in \cite{jps2} it suffices to consider only the $j=0$ integral. Using the Iwasawa decomposition as in \cite{jps2} or \cite[Cor 2.9]{km}, it suffices to prove the theorem when the integration is restricted to the torus $T_m$:
$$\sum_{r\in \ZZ}\int_{\{a\in T_m:v(\det a)=r\}}\left(W(\begin{smallmatrix}a&0\\0&I_{n-m}\end{smallmatrix})\otimes W(a)\right)X^{v(\det a)}da$$

We parametrize the torus $T_m$ by $$\prod_{i=1}^mF^{\times} \rightarrow T_m: (a_1,\dots,a_n)\mapsto \left(\begin{smallmatrix} a_1\cdots a_m & \\ &a_2\cdots a_m &\\ &&\ddots& \\ &&&a_m \end{smallmatrix}\right) =:a.$$ Consider the exterior product representation $\cW:= \cW(V,\psi)\otimes\cW(V',\psi)$ in $\Rep_R(G_n\times G_m)$. There is a natural surjection of $R$-modules
$\cW\rightarrow C^{\infty}(T_m,R)$ mapping $W\otimes W'$ to the restriction $W\left(\begin{smallmatrix}a&0\\0&I_{n-m} \end{smallmatrix}\right)\otimes W'(a)$. This map is the restriction of functions from $G_n\times G_m$ to the subgroup $T_m \overset{\Delta}{\hookrightarrow} T_m\times T_m \hookrightarrow G_n\times G_m$, where $\Delta$ denotes the diagonal embedding and $T_m\hookrightarrow G_n$ is the embedding of $T_m$ within the upper-left $m\times m$ block of $G_n$. Denote by $\cV$ the image of this restriction map, in other words the $A$-module generated by $$\{W\left(\begin{smallmatrix}a&0\\0&I_{n-m} \end{smallmatrix}\right)\otimes W'(a): W\in \cW(V,\psi), W'\in \cW(V',\psi)\},$$ the input $a$ indicating a function on $T_m$.

Let $v:F\rightarrow \ZZ$ denote the valuation on $F$. Given a function $\phi$ on $T_m$, we say that $\phi(a)\rightarrow 0$ uniformly as $v(a_i)\rightarrow \infty$ if there exists $N>0$ such that $v(a_i)\geq N$ implies $\phi(a)=0$. Define $$\cV_i:=\{\phi\in \cV: \phi(a) \rightarrow 0\text{ uniformly as } v(a_i)\rightarrow \infty\}.$$

For $i\leq m$ let $M_n(i)$ (resp. $M_m(i)$) denote the standard Levi subgroup $G_i\times G_{n-i}$ (resp. $G_i\times G_{m-i}$), and let $N_n(i)$ (resp. $N_m(i))$ denote its unipotent radical.

\begin{lemma}
Let $\theta_i$ denote the composition $\cW\rightarrow \cV\rightarrow \cV/\cV_i$. Then the submodule $\cW(N_n(i)\times N_m(i),\textbf{1})$ is contained in $\ker (\theta_i)$.
\end{lemma}
\begin{proof}
By definition $\cW(N_n(i)\times N_m(i),\textbf{1})$ is the submodule of $\cW$ generated by elements $(n,n')\phi - \phi$ where $(n,n')\in N_n(i)\times N_m(i)$ and $\phi\in \cW$. If $x\in G_n$ and and $x'\in G_m$ are any unipotent upper triangular matrices, we can apply $\psi$ to $(x,x')$ in $G_n\times G_m$ by embedding in $G_{n+m}$ as usual, so $\psi(x,x')=\psi(x)\psi(x')$. Moreover,  by the definition of $\cW$, $\phi(xg,x'g') = \psi(x,x')\phi(g,g')$ for $g\in G_n$, $g'\in G_m$. Now taking $a = \left(\begin{smallmatrix} a_1\cdots a_m & \\ &a_2\cdots a_m &\\ &&\ddots& \\ &&&a_m \end{smallmatrix}\right)\in T_m$, $n = \left(\begin{smallmatrix}I_i & y \\ 0 & I_{n-i} \end{smallmatrix}\right)\in N_n(i)$, and $n' = \left(\begin{smallmatrix}I_i & y' \\ 0 & I_{m-i} \end{smallmatrix}\right)\in N_m(i)$, we get
\begin{align*}
(n,n')\phi(a,a) = \phi(ana^{-1}a,an'a^{-1}a)
=  \psi(ana^{-1})\psi(an'a^{-1})\phi(a,a)
= \psi\left(\begin{smallmatrix} I_i & z \\ 0 & I_{n-i} \end{smallmatrix}\right)\psi\left(\begin{smallmatrix} I_i & z' \\ 0 & I_{m-i} \end{smallmatrix}\right)\phi(a,a)
\end{align*}
where $z$ is an $i\times (n-i)$ matrix whose bottom left entry is $a_iy_{i,1}$ and $z'$ is an $i\times (m-i)$ matrix whose bottom left entry is $a_iy'_{i,1}$. Therefore, this expression equals $\psi(a_iy_{i,1})\psi(a_iy'_{i,1})\phi(a,a)$. This shows that for $v(a_i)$ sufficiently large, $(n,n')\phi(a,a) - \phi(a,a)$ equals zero.
\end{proof}
\begin{lemma}
$\cW_{N_n(i)\times N_m(i),\textbf{1}} \cong J_{M_n(i)}\cW(V,\psi)\otimes J_{M_m(i)}\cW(V',\psi)$.
\end{lemma}
\begin{proof}
Let $X$ be an $(A\otimes B)[G_n\times G_m]$-module such that $N_n(i)\times N_m(i)$ acts trivially on $X$. Since $N_n(i)\times\{1\}$ and $\{1\}\times N_m(i)$ also act trivially, any $G_n\times G_m$-equivariant map
$\phi:\cW\rightarrow X$ satisfies $\phi((nW-W)\otimes W') = 0$ and $\phi(W\otimes (n'W'-W')) = 0$ for $n\in N_n(i)$, $n'\in N_m(i)$, $W\in \cW(V,\psi)$, $W'\in \cW(V',\psi)$. This shows that $\phi$ factors through the quotient maps
$$\cW\longrightarrow J_{M_n(i)}\cW(V)\otimes \cW(V')\longrightarrow J_{M_n(i)}\cW(V)\otimes J_{M_m(i)}\cW(V').$$ It follows that $J_{M_n(i)\times M_m(i)}\cW$ and $J_{M_n(i)}\cW(V)\otimes J_{M_m(i)}\cW(V')$ satisfy the same universal property.
\end{proof}

Hence, we've shown that the map $\theta_i$ factors through the Jacquet restriction $J_{M_n(i)}\cW(V,\psi)\otimes J_{M_m(i)}\cW(V',\psi)$.

Let $\rho_i(\varpi)$ denote right translation of a function by the diagonal matrix with $\varpi$ in the first $i$ diagonal entries: $$\left(\begin{smallmatrix}\varpi \\ &\ddots\\ &&\varpi\\ &&&1\\&&&&\ddots \\ &&&&&1 \end{smallmatrix}\right).$$ Note that if we're considering functions on the torus $T_m$ parametrized as $\prod_{i=1}^mF^{\times}$ as above, this translates to
$$\big(\rho_i(\varpi)\phi\big)(a_1,\dots,a_m)= \phi(a_1,\dots, a_{i-1},a_i\varpi,a_{i+1}\dots,a_m).$$ 

\begin{lemma}
Let $V$ and $V'$ be admissible and $G$-finite. Let $B_i$ be the $R$-subalgebra of $\End_R(\cV/\cV_i)$ generated by $\rho_i(\varpi)$. Then $B_i$ is finitely generated as a module over $R$.
\end{lemma}
\begin{proof}

For any $i$, the operator $\rho_i(\varpi)$ defines a linear endomorphism of the spaces $J_{M_n(i)}\cW(V,\psi)$ and $J_{M_m(i)}\cW(V',\psi)$, and so acts diagonally on their tensor product. For each $i$ it preserves the kernel of the surjective map
$J_{M_n(i)}\cW(V,\psi)\otimes J_{M_m(i)}\cW(V',\psi)\rightarrow \cV/\cV_i$ so in particular the sub-algebra of $\End_R(\cV/\cV_i)$ generated by $\rho_i(\varpi)$ can be identified with the subalgebra of $\End_R(J_{M_n(i)}\cW(V,\psi)\otimes J_{M_m(i)}\cW(V',\psi))$ generated by $\rho_i(\varpi)$.

But we have an injection
\begin{align*}
\End_{A[M_n(i)]}(J_{M_n(i)}\cW(V,\psi))\otimes \End_{B[M_m(i)]}(J_{M_m(i)}\cW(V',\psi))&\hookrightarrow\\ \End_{R[M_n(i)\times M_m(i)]}(J_{M_n(i)}\cW(V,\psi)\otimes J_{M_m(i)}\cW(V',\psi))
\end{align*} as $R$-modules, and the subalgebra $B_i$ we're considering lands inside the smaller space. 

By combining \cite[Prop 9.15]{h_bern}, \cite[Prop 9.12]{h_bern}, and \cite[Lemma 1]{bushnell_localization}, we deduce that for any standard Levi subgroup $M$, the functor $J_M$ preserves admissibility for primitive $W(k)[G]$-modules. Because $V$ is admissible and $G$-finite, there are finitely many primitive orthogonal idempotents $e$ of the Bernstein center such that $eV\neq 0$. Therefore $J_{M_n(i)}\cW(V,\psi)$ is an admissible $A[M_n(i)]$-module. It is a finite-type $A[M_n(i)]$-module by \cite[Prop 3.13(e)]{b-zI}, whose proof relies only on the fact that, if $P_n(i)$ is the parabolic subgroup $M_n(i)N_n(i)$, then $P\backslash G$ is compact. Hence we can take a finite set $\{w_i\}$ of $A[M_n(i)]$ generators and a sufficiently small compact open subgroup $U$ of $M_n(i)$ which fixes them all. Any $A[M_n(i)]$-equivariant endomorphism is uniquely determined by its values on $\{w_i\}$. On the other hand, $M_n(i)$-equivariance means such an endomorphism preserves $U$-invariance, and the $U$-fixed vectors are finitely generated, therefore it is uniquely determined via $A$-linearity from a finite set of values. This shows that the algebra $\End_{A[M_n(i)]}(J_{M_n(i)}\cW(V,\psi))$ is finitely generated as an $A$-module, hence its sub-algebra defined by $B_i$ is also finitely generated. The same is true for $J_{M_m(i)}\cW(V',\psi)$, hence their tensor product is finitely generated as a module over $A\otimes B$.
\end{proof}

Given $1\leq j \leq m$, define $\mathcal{V}^j$ (resp. $\cV^j_i$) to be the submodule of $C^{\infty}(T_j,R)$ given by $\{\phi|_{T_j}:\phi\in \mathcal{V} \text{ (resp. $\phi\in \cV_i$)}\}$.

\begin{lemma}
\label{productofmonicpolys}
There exist monic polynomials $f_1,\dots,f_m$ in $R[X]$ with unit constant term such that, for any $j=1,\dots,m$, the product $f_1(\rho(\varpi_1))\cdots f_j(\rho(\varpi_j))$ maps $\cV^j$ into $\cap_{i\leq j}\cV^j_i$.
\end{lemma}
\begin{proof}
Proving the lemma means showing that, given $W\in \cV$, there exist $N_1,\dots,N_j$ sufficiently large that $$\left(f_1(\rho_1(\varpi))\cdots f_j(\rho_j(\varpi))W\right)(a_1,\dots,a_j)= 0$$ whenever any $a_i$ satisfies $v(a_i)>N_i$, for $i\leq j$. The set $\{N_1,\dots,N_j\}$ must depend only on $W$.

We proceed by induction on $m$. If $m=1$ then $\cap_i\cV_i = \cV_1$, so this follows directly from the $R$-module finiteness of $\langle\rho_1(\varpi)\rangle \subset \End_R(\cV/\cV_1)$. The constant term is a unit because $\rho_1(\varpi)$ is invertible.

Assume the lemma is true for $m-1$. Fix $W(a_1,\dots,a_m)$ an element of $\cV$. Since $\rho_m(\varpi)$ is an integral element of the ring $\End(\cV/\cV_m)$ and $\rho_m(\varpi)$ is invertible, there exists a monic polynomial $f_m(X)$ with unit constant term such that $f_m(\rho_m(\varpi))=0$ in $\End(\cV/\cV_m)$, in other words there exists $N_m$ such that 
$\big(f_m(\rho_m(\varpi))W\big)(a_1,\dots,a_m)=0$ whenever $v(a_m)>N_m$.

Now fix $b\in F^{\times}$ and define $\phi_b:\prod_{i=1}^{m-1}F^{\times}\rightarrow R$ to be the function $$\phi_b:(a_1,\dots,a_{m-1})\mapsto \big(f_m(\rho_m(\varpi))W\big)(a_1,\dots,a_{m-1},b).$$ Note that $\phi_b\equiv 0 $ when $v(b)>N_m$.

We can apply the induction hypothesis to $\cV^{m-1}$ to conclude there exist polynomials $f_1,\dots,f_{m-1}$ in $R[X]$ satisfying the required conditions, such that for any $\phi\in \cV^{m-1}$, there are large integers $N_1(\phi),\dots,N_{m-1}(\phi)$, depending on $\phi$, such that $\Big(f_1(\rho_1(\varpi))\cdots f_{m-1}(\rho_{m-1}(\varpi))\phi\Big)(a_1,\dots,a_{m-1}) = 0$ whenever any one of $a_1,\dots, a_{m-1}$ satisfies $v(a_i)>N_i(\phi)$.

Since $\phi_b$ is the restriction of a product of Whittaker functions to $T_{m-1}$ by construction, we can apply this specifically to $\phi_b$: there exist large integers $N_1(b)$, $\dots$, $N_{m-1}(b)$, depending on $b$, such that
$$\Big(f_1(\rho_1(\varpi))\cdots f_{m-1}(\rho_{m-1}(\varpi))f_m(\rho_m(\varpi))W\Big)(a_1,\dots,a_{m-1},b) = 0$$ whenever any one of $a_1,\dots,a_{m-1}$ satisfies $v(a_i)>N_i(b)$.

We wish to show that we can choose the $N_i$'s independently of $b$. But, since $\phi_b\equiv 0$ for $v(b)>N_m$, and $\phi_b$ also vanishes when $v(b)<<0$ by Lemma \ref{finitelymanynegativeterms}, we have that $\phi_b$ is only nonzero when $b$ is confined to a compact subset of $F^{\times}$. In particular, since $f_m(\rho_m(\varpi))W$ is locally constant in each variable (in particular its last variable), there are only finitely many distinct functions $\phi_b$ as $b$ ranges over this compact set. Thus, the sets $\{N_i(b):b\in F^{\times}\}$ are finite for each $i$ and we can choose $N_i$ to be $\max\{N_i(b):b\in F^{\times}\}$.

Therefore, we have
$\Big(f_1(\rho_1(\varpi))\cdots f_m(\rho_m(\varpi))W\Big)(a_1,\dots,a_m)= 0$ whenever $v(a_i)>N_i$ for $i=1,\dots,m$, as desired.
\end{proof}

We can now deduce Theorem \ref{rationality2} as follows. Slightly abusively, we use the symbols $W$ and $W'$ to denote elements of $\cV$, so everything is already restricted to $T_m$. First we apply $\Psi(-,-,X)$ to both sides of the following equation:
$f(\rho_1(\varpi))\cdots f_m(\rho_m(\varpi))(W\otimes W') = W_0$, for $W\otimes W'\in \cV$ and $W_0\in \cap_i\cV_i$. In particular, $\Psi(W_0,X)\in R[X,X^{-1}]$, so we have a polynomial on the right hand side.

Since the integrands on the left side are functions of $T_m$, we have the transformation property 
$$\Psi(\rho_1(\varpi)^{t_1}\cdots\rho_m(\varpi)^{t_m}(W\otimes W'), X) \\ = X^{t_1+2t_2+\cdots mt_m}\Psi(W,W', X).$$ Now, given the polynomials $f_i$ in Lemma \ref{productofmonicpolys}, we can define the multivariate polynomial

$$f(X_1,\dots,X_m):=f_1(X_1)f_2(X_2)\cdots f_m(X_m)$$ in $R[X_1,\cdots,X_m]$.
Then, we have shown that $\widetilde{f}(X)\Psi(W,W',X)\in R[X,X^{-1}]$ where $\widetilde{f}$ is the image of $f$ in the map 
$R[X_1,\dots,X_m]\rightarrow R[X]$ given by $X_i\mapsto X^i$. Since $\widetilde{f}$ lies in $S$, this concludes the proof of Theorem \ref{rationality2}.

\section{Co-Whittaker modules and the integral Bernstein center}
We recall the basic properties of co-Whittaker modules and their relation with the Bernstein center.
\begin{definition}[\cite{h_whitt} 3.3]
\label{essentiallyAIGdual}
Let $\kappa$ be a field of characteristic different from $p$. An admissible smooth object $U$ in $\Rep_{\kappa}(G)$ is said to have essentially AIG dual if it is finite length as a $\kappa[G]$-module, its cosocle $\cosoc(U)$ is absolutely irreducible generic, and $\cosoc(U)^{(n)}=U^{(n)}$ (the cosocle of a module is its largest semisimple quotient).
\end{definition}

For example, let $G=GL_2(F)$, $\kappa = k$, and suppose $q$ is not congruent to $\pm 1\mod\ell$ (so that we are in the so-called banal setting). Then $\Ind_B^G\textbf{1}$ has essentially AIG dual, since its length is two, it is Whittaker type, and its largest semisimple quotient, the Steinberg representation, is absolutely irreducible generic.

Definition \ref{essentiallyAIGdual} is equivalent to $U^{(n)}$ being one-dimensional over $\kappa$ and having the property that $W^{(n)}\neq 0$ for any nonzero quotient $\kappa[G]$-module $W$. See \cite[Lemma 6.3.5]{eh} for details.

\begin{definition}[\cite{h_whitt} 6.1]
\label{defnofcowhitt}
An object $V$ in $\Rep_A(G)$ is said to be co-Whittaker if it is admissible, of Whittaker type, and $V\otimes_A \kappa(\fp)$ has essentially AIG dual for each $\fp$ in $\Spec(A)$.
\end{definition}

Co-Whittaker representations satisfy Schur's lemma \cite[Prop 6.2]{h_whitt} and are generated over $A[G]$ by a single element \cite[Lemma 1.17]{moss1}. Every admissible Whittaker-type $A[G]$-module contains a canonical co-Whittaker submodule \cite[Prop 1.18]{moss1}. Co-Whittaker modules are generated over $A[G]$ by their Schwartz functions $\cS(V)$ \cite[Lemma 6.3.5]{eh}.
 
In the setting of co-Whittaker families, the classical notion of supercuspidal support for representations over a field does not exist. However, the following result of Helm suggests a generalization of the definition of supercuspidal support:
\begin{theorem}[\cite{h_whitt}, Thm 2.2]
Let $\kappa$ be a $W(k)$-algebra that is a field and let $\Pi_1$, $\Pi_2$ be two absolutely irreducible representations of $G$ over $\kappa$ which live in the same block of the category $\Rep_{W(k)}(G)$. By Schur's lemma there are maps $f_1,f_2:\cZ\rightarrow \kappa$ giving the action of the Bernstein center on $\Pi_1$ and $\Pi_2$. Then $\Pi_1$ and $\Pi_2$ have the same supercuspidal support if and only if $f_1=f_2$
\end{theorem}
\begin{definition}[Supercuspidal Support]
Any co-Whittaker $A[G]$ module satisfies Schur's lemma, and therefore defines a map $f_V:\cZ\rightarrow \End_{A[G]}(V)\isomto A$, which is called the supercuspidal support of $V$.
\end{definition}
In \S \ref{conversetheoremsection} we show that two co-Whittaker modules have the same supercuspidal support if and only if they have the same Whittaker space.

Consider the smooth $W(k)[G]$-module 
$\fW:=\cInd_N^G\psi$. Within the category of co-Whittaker modules up to supercuspidal support, $\fW$ satisfies the following universal property:
\begin{proposition}[\cite{h_whitt}, Thm 6.3]
\label{dominance}
If $A$ is any Noetherian $W(k)$-algebra with a $\cZ$-algebra structure, $\fW\otimes_{\cZ} A$ is a co-Whittaker $A[G]$-module. Conversely, any co-Whittaker $A[G]$-module $V$ is a quotient of the representation $\fW\otimes_{\cZ,f_V}A$.
\end{proposition}
While Theorem 6.3 in \cite{h_whitt} is stated only for primitive co-Whittaker modules, it is equivalent to Proposition \ref{dominance} above because $A$ is Noetherian, and thus has finitely many connected components, meaning $f_V(e)=0$ for all but finitely many primitive idempotents $e$. Supercuspidal support defines an equivalence relation on the set of co-Whittaker modules, under which a co-Whittaker $A[G]$-module $V$ and $\fW\otimes_{\cZ,f_V}A$ are equivalent. Thus, up to supercuspidal support, co-Whittaker $A[G]$-modules are precisely the $A$-valued points of $\Spec(\cZ)$, for Noetherian rings $A$. This will be crucial in the arguments of this paper.

\section{Functional Equation and Other Properties}
\label{functionalequationsection}

In this section we prove the functional equation (Theorem \ref{mainthm2}) for co-Whittaker modules. As in \cite{moss1} we will construct the gamma factor to be what it must in order to satisfy the functional equation for one particular Whittaker function, and then use the theory of the integral Bernstein center to show that the functional equation is satisfied for all Whittaker functions. 

Particularly important is the case where the coefficient ring is reduced and $\ell$-torsion free, since in this setting all the minimal primes have characteristic zero residue fields (if $\ell$ is not a zero-divisor, it does not live in any minimal prime). We will make repeated use of the following Lemma:

\begin{lemma}
\label{tensorproductreducedflat}
If $A$ and $B$ are reduced $\ell$-torsion free $W(k)$-algebras, then $A\otimes_{W(k)}B$ is also a reduced and $\ell$-torsion free $W(k)$-algebra.
\end{lemma}
\begin{proof}
Being $\ell$-torsion free is equivalent to being flat as a module over $W(k)$. Since the tensor product of two flat modules is again flat, we have that $A\otimes_{W(k)}B$ is $\ell$-torsion free.

To show reducedness first observe that a flat $W(k)$-algebra $R$ is reduced if and only if $R\otimes_{W(k)}\cK$ is reduced, where $\cK$ denotes $\Frac(W(k))$. To see this note that $R$ embeds in the localization $S^{-1}R$ where $S=W(k)\setminus \{0\}$, and thus an element $\frac{r}{s}$ in the localization is nilpotent if and only if $r$ is nilpotent.

Applying this to the flat $W(k)$-algebra $R= A\otimes_{W(k)}B$, it suffices to prove that $(A\otimes_{W(k)}B)\otimes_{W(k)}\cK$ is reduced. But this equals
$(A\otimes_{W(k)}\cK)\otimes_{\cK}(B\otimes_{W(k)}\cK)$. We can now apply \cite[Ch 5, \S15, Thm 3]{bourbaki_ch4} which says that the tensor product of reduced algebras over a characteristic zero field is again reduced.
\end{proof}

\begin{lemma}
\label{zeta=1}
For $V$ in $\Rep_A(G_n)$ and $V'$ in $\Rep_B(G_m)$ both of Whittaker type, there exist $W$ in $\cW(V,\psi)$ and $W'$ in $\cW(V',\psi)$ such that $\Psi(W,W',X) =1$.
\end{lemma}
\begin{proof}
The proof follows that in \cite[(2.7) p.394]{jps2}. If $M$ denotes the standard parabolic of size $(m-1,1)$, let $K=GL_m(\cO_F)$ the maximal compact subgroup of $G_m$, let $Z_m$ denote the scalar matrices. Let $P_m^{(r)}$ and $M^{(r)}$ denote the subsets of matrices with determinant having valuation $r$. Recall that $G_m=MK$ \cite[3.6 Lemma]{b-zI}, and $M=P_mZ_m$. Therefore, we have $\Psi(W,W',X) = \sum_rc_r(W,W')X^r,$
where
\begin{align*}
&c_r(W,W')= \int_{K}\int_{N_m\backslash M^{(r)}}W\left(\begin{smallmatrix}mk&0\\0&I_{n-m} \end{smallmatrix}\right)\otimes W'(mk)dmdk
\end{align*}
Given $\phi\in \cInd_{N_n}^{P_n}\psi$, and $\phi'$ in $\cInd_{N_m}^{P_m}\psi$, \cite[Prop 2.8]{moss1} tells us we can choose $W$, $W'$ so that $W|_{P_n}=\phi$ and $W'|_{P_m}=\phi'$. Suppose $K'$ is a compact open subgroup of $G_m$, with $p$-power index in $K$, such that $W'$ is invariant on the right under $K'$. Take $\phi$ to be the characteristic function of the subset $P_m^{(0)}K'$ of $P_n$ (modulo $N_m$). Then if $r=0$,
$$c_r(W,W')=[K:K']\int_{N_m\backslash P_m^{(0)}}(1\otimes\phi'(p))dp,$$ and if $r>0$, $c_r(W,W')=0$. Since $[K:K']$ is a unit in $R$, we may choose $\phi'$ so that $\int_{N_m\backslash P_m^{(0)}}\phi'(p)dp$ equals $[K:K']^{-1}$.
\end{proof} 

Let $w_{n,m} = \text{diag}(I_{n-m},w_m)$ where $w_m$ is the anti-diagonal $m\times m$ matrix with $1$'s on the anti-diagonal and let $g^{\iota}:=^tg^{-1}$. For any function $W$ on $G$ we denote by $\widetilde{W}$ the function $\widetilde{W}(g) = W(w_ng^{\iota})$. If $V'$ is co-Whittaker is satisfies Schur's lemma, and therefore has a central character, which we denote $\omega$.

\begin{theorem}
\label{mainthm2}
Suppose $A$ and $B$ are Noetherian $W(k)$-algebras, and suppose $V$, $V'$ are co-Whittaker $A[G_n]$- and $B[G_m]$-modules respectively. Then there exists a unique element $\gamma(V\times V',X,\psi)$ of $S^{-1}(R[X,X^{-1}])$ such that
$$\Psi(W,W',X;j)\gamma(V\times V',X,\psi)\omega_{V'}(-1)^{n-1} = \Psi(w_{n,m}\widetilde{W},\widetilde{W'},\frac{q^{n-m-1}}{X};n-m-1-j)$$ for any $W\in \cW(V,\psi)$, $W'\in \cW(V',\psi)$ and for any $0\leq j \leq n-m-1$.
\end{theorem}
Our notation in this theorem is slightly different from that in \cite{jps2}, and follows \cite[2.1 Thm]{cogshap}.

When $A$ and $B$ are reduced and $\ell$-torsion free, and $V$, $V'$ are Whittaker type, admissible, and finitely generated over $A[G_n]$ and $B[G_m]$, respectively, formal commutative algebra allows us to reduce the functional equation to the classical setting of a characteristic zero field (Lemma \ref{intermediatefnleqn} below). By assuming further that $V$ and $V'$ are co-Whittaker, we are able to go beyond the case where $A$ and $B$ are reduced and $\ell$-torsion free by using the description of $\Spec(\cZ)$ as a moduli space for co-Whittaker modules (Proposition \ref{dominance}), together with compatibility of the gamma factor with change of base ring.

\begin{proof}
First, we define a candidate for the gamma factor. Because $V$ and $V'$ are Whittaker type, this is done exactly as in \cite[\S 4.3]{moss1}, using Lemma \ref{zeta=1} in place of \cite[Prop 4.8]{moss1}. 

Assume for now that $V$ and $V'$ are primitive co-Whittaker modules. We will remove this assumption at the end of the proof.

Second, we prove the following intermediate Lemma:
\begin{lemma}
\label{intermediatefnleqn}
Suppose $A$ and $B$ are reduced $\ell$-torsion free $W(k)$-algebras, and $V$, $V'$ are Whittaker type, admissible, and finitely generated over $A[G_n]$ and $B[G_m]$, respectively. Then the functional equation of Theorem \ref{mainthm2} holds.
\end{lemma}
\begin{proof}
The arguments of \cite[Thm 4.11]{moss1} carry over to this setting. Since the zeta integrals $\Psi(W,W',X;j)$ all live in $S^{-1}R[X,X^{-1}]$ we can make sense of both sides of the functional equation. Any characteristic zero point $R\rightarrow \kappa$ gives characteristic zero points of $A$ and $B$. Fix an algebraic closure $\overline{\kappa}$ and choose an embedding $\CC\hookrightarrow \overline{\kappa}$. By applying the arguments of \cite{jps2} to $V\otimes\overline{\kappa}$ and $V'\otimes\overline{\kappa}$, where $X$ replaces the variable $q^{-s+\frac{n-m}{2}}$, we find that the image of $$\Psi(W,W',X;j)\gamma(V\times V',X,\psi)\omega_{V'}(-1)^{n-1} -\Psi(w_{n,m}\widetilde{W},\widetilde{W'},\frac{q^{n-m-1}}{X};n-m-1-j)$$ in $R\rightarrow \kappa$ equals zero. Using Lemma \ref{tensorproductreducedflat}, we have that $R$ is reduced and $\ell$-torsion free, hence its minimal primes have characteristic zero residue fields, and this difference  equals zero modulo every minimal prime. By reducedness, the minimal primes have trivial intersection.
\end{proof}

Third, we focus on removing the hypothesis that $A$ is reduced and $\ell$-torsion free. To do this we mimic the argument of \cite[\S 5]{moss1}, and therefore content ourselves with a sketch. Let $\cZ$ be the center of $\Rep_{W(k)}(G_n)$, let $\cZ'$ be the center of $\Rep_{W(k)}(G_m)$. It is proved in \cite[Thm 10.8]{h_bern} that for primitive idempotents $e$ and $e'$ in $\cZ$ and $\cZ'$ respectively, $e\cZ$ and $e'\cZ'$ are reduced and $\ell$-torsion free $W(k)$-algebras. Lemma \ref{tensorproductreducedflat} implies that $e\cZ\otimes_{W(k)}e'\cZ'$ is reduced and $\ell$-torsion free, so in particular the hypotheses of the theorem hold for the pair of representations $e\fW_n$ and $e'\fW_m$. We may therefore define the universal gamma factor $\gamma(e\fW_n\times e'\fW_m,X,\psi)\in S^{-1}(e\cZ\otimes e'\cZ')[X,X^{-1}]$.

Now, given primitive co-Whittaker modules $V$ in $e\Rep_{W(k)}(G_n)$ and $V'$ in $e'\Rep_{W(k)}(G_m)$ over any coefficient rings $A$ and $B$ which are Noetherian $W(k)$-algebras, we have supercuspidal supports $f_V:e\cZ\rightarrow A$ and $f_{V'}:e'\cZ'\rightarrow B$ such that $e\fW_n\otimes_{e\cZ,f_V}A$ dominates $V$ and $e'\fW_m\otimes_{e'\cZ',f_{V'}}B$ dominates $V'$.

Because the formation of zeta integrals and gamma factors commute with change of base ring \cite[\S 4.2]{moss1}, the image of $\gamma(e\fW_n\times e'\fW_m,X,\psi)$ in the map $S^{-1}(e\cZ\otimes e'\cZ')[X,X^{-1}]\rightarrow S^{-1}R[X,X^{-1}]$ induced by $f_V\otimes f_{V'}$ equals $\gamma(V\times V',X,\psi)$.

Since $e\fW_n\otimes_{e\cZ,f_V}A$ dominates $V$, they have the same Whittaker spaces, and thus share all the same zeta integrals, and the same goes for $V'$. Therefore, $\gamma(V\times V',X,\psi)$ satisfies the functional equation for all $W\in \cW(V,\psi)$ and all $W'$ in $\cW(V',\psi)$.

We now remove the hypothesis that $V$ and $V'$ are primitive. Since $A$ and $B$ are Noetherian, $V$ (resp. $V'$) is a finite direct sum of representations $e_iV$ (resp. $f_iV'$) which are co-Whittaker $e_iA[G]$ (resp. $f_iB[G]$)-modules, where $e_i$ (resp. $f_i$) are distinct primitive orthogonal idempotents of $\cZ$ (resp. $\cZ'$, the center of $\Rep_{W(k)}(G_m)$). By abuse of notation we are identifying the $e_i$'s and $f_i'$ with their images $f_V(e_i)$ and $f_{V'}(f_i)$ in $A$ and $B$, respectively. If $W$ (resp. $W'$ is in $\cW(V,\psi)$ (resp. $\cW(V',\psi)$), then $W = \sum_i e_iW$ and $W' = \sum_i f_iW$. The functions $\Psi(-,-,X;j)$ and $\Psi(w_{n,m}\widetilde{(-)},\widetilde{(-)}, \frac{q^{n-m-1}}{X};n-m-1-j)$ are linear in each input. It now follows from the primitive case, together with the fact that for $i\neq j$ $e_ie_j\otimes 1$ and $1\otimes f_if_j$ equal zero in $S^{-1}R[X,X^{-1}]$, that $\Psi(W,W',X)$ satisfies a functional equation as in Theorem \ref{mainthm2}, with functional constant given by $\sum_{i,j}\gamma(e_iV\times f_j V', X, \psi)$. Since there are finitely many components, we can make the identification $$S^{-1}R[X,X^{-1}] = \bigoplus_{i,j}S_{ij}^{-1}(e_iA\otimes f_jB)[X,X^{-1}],$$ where $S_{ij}$ denotes the subset of $(e_iA\otimes f_jB)[X,X^{-1}]$ consisting of polynomials whose first and last coefficients are units. It follows that $\gamma(V\times V',X,\psi)$ is in $S^{-1}R[X,X^{-1}]$, and uniqueness follows from uniqueness at each component.
\end{proof}

We can now define a universal gamma factor. If $\fW_n:=\cInd_{N_n}^{G_n}\psi$ and $\fW_m:=\cInd_{N_m}^{G_m}\psi$, we form the sum
$$\gamma(\fW_n\times\fW_m,X,\psi) :=\sum_{e,e'}\gamma(e\fW_n\times e'\fW_m,X,\psi)\in (\cZ\otimes \cZ')[[X]][X^{-1}],$$ where the sum runs over all primitive idempotents of $\cZ$ and $\cZ'$, respectively (here $\cZ'$ is the center of $\Rep_{W(k)}(G_m)$). Since $\cZ$ and $\cZ'$ are not Noetherian, we do not have a functional equation, or even rationality, for $\gamma(\fW_n\times\fW_m,X,\psi)$, but we have the following universal property:
\begin{corollary}
\label{universalgamma}
Let $A$, $B$ be any Noetherian $W(k)$-algebras, let $V$ be a co-Whittaker $A[G_n]$-module and $V'$ a co-Whittaker $B[G_m]$-module, having supercuspidal supports $f_V$ and $f_{V'}$. Then $$\gamma(V\times V',X,\psi) = (f_V\otimes f_{V'})\big(\gamma(\fW_n\times\fW_m,X,\psi)\big).$$
\end{corollary}
For any $A[G]$-module $V$, let $V^{\iota}$ denote the $A[G]$-module whose underlying $A$-module is given by $V$ and on which $G$ acts by pre-composition with the involution $g\mapsto g^{\iota}$. Following \cite[Lemma 4.4]{moss1}, $V^{(n)}$ is isomorphic to $(V^{\iota})^{(n)}$, from which it follows immediately that $V^{\iota}$ is co-Whittaker, with Whittaker space $\{\widetilde{W}:W\in \cW(V,\psi)\}$.
\begin{corollary}
\label{gammaisaunit}
For $V\in \Rep_A(G_n)$ and $V'\in \Rep_B(G_m)$ co-Whittaker modules and $A$, $B$ any Noetherian $W(k)$-algebras, $\gamma(V\times V',X)$ is a unit in $S^{-1}(R[X,X^{-1}])$ and $\gamma(V\times V',X,\psi)^{-1} = \gamma(V^{\iota}\times(V')^{\iota},\frac{q^{n-m-1}}{X},\psi^{-1})$.
\end{corollary}
\begin{proof}
Let $W$ and $W'$ be the Whittaker functions guaranteed by Lemma \ref{zeta=1}. The functional equation reads $$\Psi(W,W',X)\gamma(V\times V', X,\psi)\omega_{V'}(-1)^m = \Psi(\widetilde{W},\widetilde{W'},\frac{q^{n-m-1}}{X}).$$ Replacing $X$ with $\frac{q^{n-m-1}}{X}$ we have
$$\Psi(W,W',\frac{q^{n-m-1}}{X})\gamma(V\times V', \frac{q^{n-m-1}}{X},\psi)\omega_{V'}(-1)^m \\= \Psi(\widetilde{W},\widetilde{W'},X).$$ Now multiplying through by $\gamma(V^{\iota}\times(V')^{\iota},X,\psi^{-1})\omega_{(V')^{\iota}}(-1)^m$ and noticing that  $\omega_{(V')^{\iota}}=\omega_{V'}^{-1}$ we get:
\begin{align*}
\Psi(W,W',\frac{q^{n-m-1}}{X})\gamma(V\times V', \frac{q^{n-m-1}}{X},\psi)\gamma(V^{\iota}\times(V')^{\iota},X,\psi^{-1}) = \Psi(W,W',\frac{q^{n-m-1}}{X}),
\end{align*}
By Lemma \ref{zeta=1} we have $\gamma(V\times V', \frac{q^{n-m-1}}{X},\psi)\gamma(V^{\iota}\times(V')^{\iota},X,\psi^{-1}) = 1$.
\end{proof}

\section{A Converse Theorem for $GL(n)\times GL(n-1)$}
\label{conversetheoremsection}

The main result of this chapter is that the collection of gamma factors of pairs uniquely determines the supercuspidal support of a co-Whittaker family.

\begin{theorem}
\label{converse}
Let $A$ be a finite-type $W(k)$-algebra which is reduced and $\ell$-torsion free, and let $\cK=\Frac(W(k))$. Suppose $V_1$ and $V_2$ are two co-Whittaker $A[G_n]$-modules. There is a finite extension $\cK'$ of $\cK$ with ring of integers $\cO$ such that, if $\gamma(V_1\times V',X,\psi) = \gamma(V_2\times V',X,\psi)$ for all co-Whittaker $\cO[G_{n-1}]$-modules $V'$, then $\cW(V_1,\psi)=\cW(V_2,\psi)$.
\end{theorem}
We must suppose that $A$ is a finite-type, reduced, and $\ell$-torsion free in order to implement Theorem \ref{rikka2}. The geometric methods used to prove Theorem \ref{rikka2} require $A$ to be reduced in order that a certain subset of desirable points is open (Lemmas \ref{torsionfree}, \ref{Bexists}); they require $A$ to be finite and flat over $W(k)$ so that this subset defines open subset of $\Spec(\cZ)$.

Because of the control achieved in Theorem \ref{rikka2}, it suffices to take in the statement of Theorem \ref{converse} only those co-Whittaker modules $V'$ such that $V'\otimes_{\cO}\cK'$ is absolutely irreducible, which gives Theorem \ref{intromainthm}.

If $V_1$ and $V_2$ live in the same block of the category $\Rep_{W(k)}(G)$, the argument of Proposition \ref{Wpointsdense} need only be applied to the primitive idempotent $e$ associated to this block, and the finite extension $\cK'$ depends only on this block.

The following Propostion shows the relationship between Whittaker models and supercuspidal support.
\begin{proposition}
\label{whittakerequalssupport}
Suppose $V_1$ and $V_2$ are co-Whittaker $A[G]$-modules. Then $\cW(V_1,\psi)=\cW(V_2,\psi)$ if and only if $f_{V_1}\equiv f_{V_2}$.
\end{proposition}
\begin{proof}
It follows from Lemma \ref{dominanceclassequalssupport} below that $f_{V_1} = f_{\cW(V_1,\psi)} = f_{\cW(V_2,\psi)} = f_{V_2}$. For the converse, note that by \cite[Lemma 2.6]{moss1} we have $\cW(V_1,\psi)=\cW(\fW\otimes_{\cZ,f_{V_1}}A,\psi)=\cW(\fW\otimes_{\cZ,f_{V_2}}A,\psi)=\cW(V,\psi)$.
\end{proof}

\begin{lemma}
\label{dominanceclassequalssupport}
Suppose we have two co-Whittaker modules $V_1$ and $V_2$ such that $V_2$ is a quotient of $V_1$. Then $f_{V_1}\equiv f_{V_2}$. 
\end{lemma}
\begin{proof}
Suppose $\phi:V_1\rightarrow V_2$ is any surjective $G$-equivariant map. Choosing a cyclic $A[G]$-generator $v_1\in V_1$, then $\phi(v_1)$ is an $A[G]$-generator of $V_2$ since its image in $V_2^{(n)}$ is a generator \cite[Lemma 2.29]{moss1}. Denote by $v_1'$ the image of $v_1$ in $V_1\rightarrow V_1^{(n)}$. We have $v_1'$ generates $V_1^{(n)}$ and $\phi^{(n)}(v_1')$ generates $V_2^{(n)}$.

If $z$ is an element of $\cZ$, then $z_{V_1}\in \End_G(V_1)$ sends $v_1$ to $f_{V_1}(z)v_1$, where $f_{V_1}(z)\in A$. By definition, the action of the Bernstein center is functorial, hence commutes with the morphism $\phi$, thus $$z_{V_2}(\phi(v_1)) = \phi(f_{V_1}(z)v_1) = f_{V_1}(z)\phi(v_1).$$ Since $\phi(v_1)$ is an $A[G]$-generator of $V_2$, $z_{V_2}$ is completely determined by where it sends $\phi(v_1)$. This shows that the map $f_{V_2}:\cZ\rightarrow \End_G(V_2)\rightarrow A$ given by $z\mapsto z_{V_2}\mapsto f_{V_2}(z)$ exactly equals the map $f_{V_1}$.
\end{proof}

%

We remark that any nonzero $G$-equivariant homomorphism between co-Whittaker modules is a surjection.

\subsection{Proof of converse theorem}

For two $W(k)$-algebras $A$, $B$, $\phi_1\in\cInd_N^G\psi_A$ and $\phi_2\in\Ind_N^G\psi_B^{-1}$ we denote by $\langle \phi_1,\phi_2\rangle$ the element $$\int_{N\backslash G}\phi_1(x)\otimes\phi_2(x)dx\ \in A\otimes_{W(k)}B$$ and let $\cK=\Frac{W(k)}$. At the heart of the proof of the converse theorem will lie the following result, whose proof is postponed until in \S \ref{proofofvanishingthm}

\begin{theorem}
\label{rikka2}
Suppose $A$ is a finite-type, reduced, $\ell$-torsion free $W(k)$-algebra. Suppose $H\neq 0$ is an element of $\cInd\psi_A$. Then there exists a finite extension $\cK'$ of $\cK$ with ring of integers $\cO$ and an absolutely irreducible generic integral $\cK'$ representation $U'$ with integral structure $U$, such that there is a Whittaker function $W\in \cW(U^{\vee},\psi_{\cO}^{-1})$ satisfying $\langle H,W\rangle \neq 0$ in $A\otimes_{W(k)}\cO$.
\end{theorem}

The rest of this section is devoted to proving Theorem \ref{converse}, while assuming Theorem \ref{rikka2}. 

Let $m=n-1$ and let $V_1$ and $V_2$ be co-Whittaker $A[G_n]$-modules. For $i=1,2$, choose $G$-homomorphisms $\omega_i:V_i\rightarrow \Ind_N^G\psi$ generating $\Hom_G(V_i,\Ind_N^G\psi)$. Let $\cS(V_i)$ denote the sub-$A[P_n]$-module of $V_i$ given by the Schwartz functions of $V_i$ (Definition \ref{schwartzdefn}). 

\begin{lemma}
\label{restrictionschwartzfunctions}
Consider the sub-$A[P_n]$-modules $\omega_i(\cS(V_i))$ of $\Ind_N^G\psi$. If $r_P:\Ind_N^G\psi\rightarrow \Ind_N^P\psi$ denotes the map given by restriction of functions, then $r_P(\omega_1(\cS(V_1))) = r_P(\omega_2(\cS(V_2)))$.
\end{lemma}
\begin{proof}
Let $\omega_{i,P}$ be the maps $V_i|_P\rightarrow \Ind_N^P\psi$ guaranteed by genericity. Then we have $r_P\circ \omega_i = \omega_{i,P}$ from the definitions. By \cite[Proposition 2.8 (2)]{moss1}, we have $\omega_{1,P}(\cS(V_1))=\omega_{2,P}(\cS(V_2)) = \cInd_N^P\psi$ as subsets of $\Ind_N^P\psi$. This proves the claim.
\end{proof}

\begin{proposition}
\label{whittakerfunctionsagreeonG}
Suppose $A$ is reduced, flat, and finite type over $W(k)$, and suppose the gamma factors are equal, as in Theorem \ref{converse}. Take $W_1\in \omega_1(\cS(V_1))$ and $W_2\in \omega_2(\cS(V_2))$ such that $r_P(W_1) = r_P(W_2)$, then $W_1= W_2$ as elements of $\Ind_N^G\psi$.
\end{proposition}
\begin{proof}
The proof follows \cite{hen_converse}. The assumptions on $A$ will allow us to invoke Theorem \ref{rikka2}.

Let $\fS$ be the subspace of $\cW(V_1,\psi)\times \cW(V_2,\psi)$ consisting of pairs $(W_1,W_2)$ such that $r_{G_m}(W_1)=r_{G_m}(W_2)$, where $r_{G_m}$ denotes restriction to the subgroup $G_m$ of $G_n$ (with $m=n-1$). Since $G_m\subset P_n$, Lemma \ref{restrictionschwartzfunctions} shows this is nonempty. Let $(W_1,W_2)\in \fS$. Then
$\Psi(W_1,W',X) = \Psi(W_2,W',X)$ for all $W'\in \cW(V',\psi_{\cO}^{-1})$ as $V'$ varies over all co-Whittaker $\cO[G_{n-1}]$-modules. By assumption, $\gamma(V_1\times V', X, \psi) = \gamma(V_2\times V',X,\psi)$ for all such $V'$, hence the equality of the products:
$\Psi(W_1,W',X)\gamma(V_1\times V', X, \psi) = \Psi(W_2,W',X)\gamma(V_2\times V',X,\psi)$. Applying the functional equation with $j=0$ and $m= n-1$ we thus conclude that
$\Psi(\widetilde{W_1},\widetilde{W'},\frac{q^{n-m-1}}{X})=\Psi(\widetilde{W_2},\widetilde{W'},\frac{q^{n-m-1}}{X})$, and furthermore $\Psi(\widetilde{W_1},\widetilde{W'},X)=\Psi(\widetilde{W_2},\widetilde{W'},X)$.

For each integer $m$, denote by $H_m$ the function on $G_m$ given by
\begin{eqnarray*}
H_m(g) = 0 & \text{if} & v_F(\det g)\neq m\\
H_m(g) = \widetilde{W_1}\left(\begin{smallmatrix}g & 0 \\ 0 & 1 \end{smallmatrix}\right) - \widetilde{W_2}\left(\begin{smallmatrix}g & 0 \\ 0 & 1 \end{smallmatrix}\right) & \text{if} & v_F(\det g)=m.
\end{eqnarray*} Then the equality of formal Laurent series $\Psi(\widetilde{W_1},\widetilde{W'},X)=\Psi(\widetilde{W_2},\widetilde{W'},X)$ implies that, for each $m$, we have $$\int_{N_m\backslash G_m}H_m(g)\otimes\widetilde{W'}(g)dg = 0$$ for all $W'$ in the Whittaker spaces $\cW(V',\psi_{\cO})$ of all co-Whittaker $\cO[G]$-modules $V'$.

Now suppose $V'$ has the property that $V'\rightarrow V'\otimes_{\cO}\cK'$ is an embedding and $V'\otimes\cK'$ is absolutely irreducible. Then $(V')^{\vee}\otimes\cK'\cong(V'\otimes\cK')^{\vee}$, and by \cite[Thm 7.3]{b-z}, $(V'\otimes\cK')^{\vee}\cong (V'\otimes\cK')^{\iota}$. Thus $$\cW((V'\otimes\cK')^{\vee},\psi_{\cK'}^{-1})=\cW((V'\otimes\cK')^{\iota},\psi_{\cK'}^{-1})=\cW((V')^{\iota},\psi_{\cO}^{-1})\otimes \cK'.$$ So given $W^{\vee}\in \cW((V')^{\vee},\psi_{\cO}^{-1})$, we can view $W^{\vee}$ as an element of $\cW((V')^{\iota},\psi_{\cO}^{-1})[\frac{1}{\varpi}]$ where $\varpi$ is a uniformizer of $\cO$. In other words, there is an integer $s$ such that $\varpi^sW^{\vee}$ is given by an element $\widetilde{W}$ in $\cW((V')^{\iota},\psi_{\cO}^{-1})$. Therefore
$$\varpi^s\langle H_m,W^{\vee}\rangle = \langle H_m,\varpi^sW^{\vee}\rangle=\langle H_m,\widetilde{W}\rangle = 0,$$ which implies $\langle H_m,W^{\vee}\rangle=0$ since $A\otimes_{W(k)}\cO$ is flat over $\cO$ (i.e. $\varpi$-torsion free). Therefore we can apply the contrapositive of Theorem \ref{rikka2} to conclude that each $H_m$ is identically zero, for all $m$. Hence $\widetilde{W_1}\left(\begin{smallmatrix}g & 0 \\ 0 & 1 \end{smallmatrix}\right)\equiv \widetilde{W_2}\left(\begin{smallmatrix}g & 0 \\ 0 & 1 \end{smallmatrix}\right)$.

Let $\widetilde{\fS}$ be the subspace of $\cW(V_1^{\iota},\psi^{-1})\times\cW(V_2^{\iota},\psi^{-1})$ consisting of pairs $(U_1,U_2)$ whose restrictions to $G_m\subset G_n$ are equal. Then we have shown that $(\widetilde{W_1},\widetilde{W_2})\in \widetilde{\fS}$. In fact, the following result is true:
\begin{lemma}
\label{tildeonspaceoffunctionsagreeingonP}
Let $W_1$ be in $\cW(V_1,\psi)$ and $W_2$ be in $\cW(V_2,\psi)$. Then $(W_1,W_2)$ is in $\fS$ if and only if $(\widetilde{W_1},\widetilde{W_2})$ is in $\widetilde{\fS}$.
\end{lemma}
\begin{proof}[Proof of Lemma \ref{tildeonspaceoffunctionsagreeingonP}]
We have just proved one direction. By Lemma \ref{gammaisaunit}, our hypothesis on the equality of gamma factors is equivalent to the equality of the gamma factors $\gamma(V_1^{\iota}\times(V')^{\iota},X,\psi^{-1})=\gamma(V_2^{\iota}\times(V')^{\iota},X,\psi^{-1})$ for all $(V')^{\iota}$. Since $(-)^{\iota}$ is an exact covariant functor which is additive in direct sums, commutes with base-change, and induces an isomorphism between Whittaker spaces, $V\mapsto V^{\iota}$ preserves the property of being co-Whittaker and $V^{\iota}$, $(V')^{\iota}$ are again co-Whittaker. Thus the entire argument works replacing $V_i$ with $V_i^{\iota}$ and $V'$ with $(V')^{\iota}$ to get the converse implication.
\end{proof}

We now continue with the proof of Proposition \ref{whittakerfunctionsagreeonG}. If we let $G_n$ act diagonally on $\cW(V_1,\psi)\times \cW(V_2,\psi)$ and on $\cW(V_2^{\iota},\psi^{-1})\times \cW(V_2^{\iota},\psi^{-1})$, both by right translation, then the subgroup $P_n$ stabilizes the subspaces $\fS$ and $\widetilde{\fS}$. To see this note that for $g\in G_m$ and $u\in U_n$ we have $W_i(gu) = W_i(gug^{-1}g) = \psi^g(u)W_i(g)$, so $uW_i$'s restriction to $G_m$ is completely determined. If $\rho$ denotes right translation, a short calculation shows $\rho(g^{\iota})\widetilde{W}(x) = \widetilde{gW}(x)$. Combining this with the lemma above, it follows that $\fS$ is stable under $^tP$ as well. Hence $\fS$ is stable under the group generated by $P$ and $^tP$. But this group contains all elementary matrices, hence contains all of $SL_n(F)$. On the other hand, this group also contains matrices of any determinant. Hence for any $a\in F^{\times}$ it contains all matrices in $GL_n(F)$ with determinant $a$; in other words this group equals $G$.

Therefore $\fS$ is stable under the action of all of $G_n$. Given $W_1$ and $W_2$ such that $r_P(W_1)=r_P(W_2)$ we have that $r_P(gW_1) = r_P(gW_2)$ for any $g\in G_n$ so we have $gW_1(1)=gW_2(1)$, i.e. $W_1(g)=W_2(g)$ for all $g\in G_n$. This concludes the proof of Proposition \ref{whittakerfunctionsagreeonG}.
\end{proof}

\begin{corollary}
\label{equalschwartzfunctions}
If $A$, $V_1$, and $V_2$ satisfy the hypotheses of Theorem \ref{converse}, then $\omega_1(\cS(V_1)) = \omega_2(\cS(V_2))$.
\end{corollary}
\begin{proof}
Given $W_1$ in the left side, there exists $W_2$ such that $r_P(W_1)=r_P(W_2)$. Proposition \ref{whittakerfunctionsagreeonG} then implies $W_1=W_2\in \omega_2(\cS(V_2))$ which shows one containment. The argument to show the opposite containment is identical.
\end{proof}

We can now deduce Theorem \ref{converse}. Suppose $A$, $V_1$, $V_2$ satisfy the hypotheses of Theorem \ref{converse}. Since $V_i$ is co-Whittaker and surjects onto $\cW(V_i,\psi)$, we have that $\cW(V_i,\psi)$ is also co-Whittaker. Moreover, $\cS(\cW(V_i,\psi))=\omega_i\cS(V_i)$ as $A[P]$-modules. By Corollary \ref{equalschwartzfunctions}, $\omega_1(\cS(V_1)) = \omega_2(\cS(V_2))$ and hence live in $\cW(V_1,\psi)\cap \cW(V_2,\psi)$, the intersection taken within $\Ind_N^G\psi$. Because the $\cW(V_i,\psi)$ are co-Whittaker, it  follows from \cite[Lemma 6.3.2]{eh} that the $A$-submodule $\cS(\cW(V_i,\psi))=\omega_i\cS(V_i)$ generates $\cW(V_i,\psi)$ as an $A[G]$-module.  Hence there is a subset of $\cW(V_1,\psi)\cap\cW(V_2,\psi)$ which generates both $\cW(V_i,\psi)$ and $\cW(V_i,\psi)$ as $A[G]$-modules, and so $\cW(V_1,\psi)=\cW(V_2,\psi)$. This concludes the proof of Theorem \ref{converse}.


\subsection{Generic irreducibility}

In this subsection we prove that co-Whittaker families are ``generically'' irreducible, in the sense of algebraic geometry. This sort of property is widely known in representation theory, but must be verified in this setting, as it will be used later in the proof of Theorem \ref{rikka2}. Let $e$ be a primitive idempotent of $\Spec(e\cZ)$. Since minimal primes are the generic points of irreducible components of $\Spec(e\cZ)$, the following proposition shows that $\fW|_{\fp}$ is irreducible at all points $\fp$ in a Zariski open dense subset of $\Spec(\cZ)$:
\begin{proposition}
\label{univcowhittirreducibleatminimalprimes}
Let $e$ be a primitive idempotent of $\cZ$, and suppose $\fp$ is a minimal prime ideal of $e\cZ$. Then $e\fW\otimes_{e\cZ}\kappa(\fp')$ is absolutely irreducible for all $\fp'$ in an open neighborhood of $\fp$.
\end{proposition}

\begin{proof}
Let $\Pi:=e(\cInd\psi)$. We begin by showing that the locus of points $\fp'$ such that $\Pi\otimes\kappa(\fp')$ is reducible is contained in a closed subset. For a ring $R$ and $K$ a compact open subgroup let $\cH(G,K,R)$ be the algebra of smooth compactly supported functions $G\rightarrow R$ which are $K$-fixed under left and right translations (see \cite[I.3]{vig}). $\cH(G,K,e\cZ)$ and $\Pi$ form sheaves of $\Spec(e\cZ)$-modules, and following \cite[IV.1.2]{b_rum}, the map $P_K:\cH(G,K,e\cZ)\rightarrow \End_{e\cZ}(\Pi^K)$ which sends $h$ to $\Pi(h)$ is a morphism of sheaves. We will require the following:

\begin{lemma}
\label{irreducibilityandKfixedvectors}
Let $R$ be a commutative ring with unit, let $G$ be a locally profinite group with a cofinal system $\Omega$ of compact open subgroups whose pro-order is invertible in $R$, and let $V$ be a smooth $R[G]$-module. Then $R$ is simple as an $R[G]$-module if and only if, for any $K$ in $\Omega$, $V^K$ is either zero or simple as an $\cH(G,K,R)$-module.
\end{lemma}
\begin{proof}
Recall that $\cH(G,R)$ acts on $\Rep_R(G)$ and contains, for any $K$ in $\Omega$, a projector $e_K$ taking $V$ to the $\cH(G,K,R)$-module $V^K$ (\cite[I.4.4, I.4.5]{vig}). The functor $V\mapsto V^K$ is exact (\cite[I.4.6]{vig}). Therefore, we may use the proof of \cite[4.3 Corollary]{bh} verbatim.
\end{proof}

Continuing with the proof of Proposition \ref{univcowhittirreducibleatminimalprimes}, suppose $\Pi|_{\fp'}$ is reducible. Then by Lemma \ref{irreducibilityandKfixedvectors}, there exists a $K$ such that $(\Pi|_{\fp'})^K$ is nonzero and reducible. Since $(\Pi|_{\fp'})^K$ is a finite dimensional $\kappa(\fp')$ vector space, a proper $\cH(G,K,\kappa(\fp'))$-stable subspace $Y$ gives  $\{\phi\in \End_{\kappa(\fp')}(\Pi^K):\phi(Y)\subset Y\}$, which is a proper submodule of $\End_{\kappa(\fp')}(\Pi^K)$ containing the image of $P_K\otimes\kappa(\fp')$. The set of points $\fp$ where $(P_K)_{\fp'}$ fails to be surjective is contained in the support of the finitely generated $e\cZ$-module $\frac{\End(\Pi^K)}{\Img(P_K)}$, which is closed. For any such point $\fp$, $(\Pi|_{\fp'})^K=(\Pi^K)|_{\fp'}$ must then be reducible by Schur's lemma.  Now it is only left to show that in each irreducible component of $\Spec(e\cZ)$ there is at least one point where we have irreducibility (for that point then lives in an open neighborhood of irreducible points that contains the generic point, i.e. the minimal prime).

Suppose $e = e_{[L,\pi]}$ is the idempotent corresponding to the mod $\ell$ inertial equivalence class $[L,\pi]$ in the Bernstein decomposition of $\Rep_{W(k)}(G)$ (see \cite{h_whitt}). By \cite[Prop 11.1]{h_bern}, $e\cZ\otimes_{W(k)}\overline{\cK} \cong \prod_{M,\pi'}\cZ_{\overline{\cK},M,\pi'}$ where $M,\pi'$ runs over inertial equivalence classes of $\Rep_{\overline{\cK}}(G)$ whose mod $\ell$ inertial supercuspidal support equals $(L,\pi)$, and $\cZ_{\overline{\cK},M,\pi'}$ denotes the center of $\Rep_{\overline{\cK}}(G)_{M,\pi'}$. The ring $\cZ_{\overline{\cK},M,\pi'}$ is a Noetherian normal domain. Since $e\cZ$ is reduced and $\ell$-torsion free, none of its minimal primes contain $\ell$. Inverting $\ell$, this decomposition gives isomorphisms
$$\prod_{\fp\text{ minimal}}(e\cZ/\fp)\otimes_{W(k)}\overline{\cK}\cong e\cZ\otimes_{W(k)}\overline{\cK} \cong \prod_{M,\pi'}\cZ_{{\overline{\cK}},M,\pi'}.$$
But $(e\cZ/\fp)\otimes_{W(k)}\overline{\cK}$ and $\cZ_{{\overline{\cK}},M,\pi'}$ are domains, so cannot factor as direct products, and therefore neither decomposition is finer than the other. In particular, for each minimal prime there exists $M,\pi'$ such that $(e\cZ/\fp)\otimes_{W(k)}\overline{\cK}\cong \cZ_{\overline{\cK},M,\pi'}$. Hence the algebraic closure $\overline{\kappa(\fp)}$ of $\kappa(\fp)$ is $\Frac(\cZ_{\overline{\cK},M,\pi'})$.

Given such an $M,\pi'$, we have by \cite{bd} that $$\cZ_{\overline{\cK},M,\pi'}:=\cZ(\Rep_{\overline{\cK}}(G)_{M,\pi'}) \cong (\overline{\cK}[M/M^{\circ}]^H)^{W(\pi')},$$ where $M^{\circ}$ is the subgroup generated by all the compact subgroups (which equals the set of $m\in M$ with $\det m\in U_F$). The linear algebraic group over $\overline{\cK}$ of unramified characters of $M$ acts on $\Rep_{\overline{\cK}}(M)$, and can be identified with $\Spec(\overline{\cK}[M/M^{\circ}])$. $H$ denotes the finite subgroup stabilizing $\pi'$. The Weyl group of $G$ also acts on $\Rep_{\overline{\cK}}(G)$ and $W(\pi')$ is the subgroup stabilizing $\pi'$ up to twisting by unramified characters. See \cite[Section 3]{h_bern} for a nice summary. 

By \cite[Theorem 27]{b_rum}, if $\pi'$ is our given supercuspidal representation of $\cK$, then $i_P^G(\pi'\otimes \chi)$ is absolutely irreducible for $\chi$ a generic $\overline{\cK}$ point of $\overline{\cK}[M/M^{\circ}]$. Let $\fq$ be a point of $e\cZ$ lying under the point $\chi$. Since $\pi'\otimes\chi$ is cuspidal, $(i_P^G(\pi'\otimes\chi))^{(n)}$ is one-dimensional and therefore we have a map $e(\cInd\psi)\otimes_{e\cZ}\kappa(\fq)\rightarrow i_P^G(\pi'\otimes\chi)$ coming from reciprocity. Since $i_P^G(\pi'\otimes\chi)$ is absolutely irreducible this map is surjective. The kernel $K$ of this map must be zero by the following reasoning. By \cite[Cor 3.2.14]{eh} all the Jordan-Holder constituents of an essentially AIG representation over $\overline{\cK}$ have the same supercuspidal support, so the same is true for representations with essentially AIG dual. Therefore, if $K$ were nonzero it would have all Jordan-Holder constituents having the same supercuspidal support as $i_P^G(\pi'\otimes\chi)$, in particular those constituents would be irreducible and equivalent to $i_P^G(\pi'\otimes\chi)$. But then $K^{(n)}$ is nonzero, which contradicts the fact that $e(\cInd\psi)\otimes_{\cZ_{M,\pi'}}\kappa(\fq)\rightarrow i_P^G(\pi'\otimes\chi)$ is a $G$-surjection of Whittaker type representations. Hence $e(\cInd\psi)\otimes\kappa(\fq)$ is absolutely irreducible. 
\end{proof}

\subsection{Proof of the vanishing theorem}
\label{proofofvanishingthm}
This section is devoted to the proof of Theorem \ref{rikka2}. We denote $\psi_A$ by $\psi\otimes_{W(k)} A$, then $\cInd_N^G\psi_A \cong (\cInd_N^G\psi)\otimes_{W(k)}A$. 

There exists a primitive idempotent $e$ in $\cZ$ such that $eH\neq 0$. Moreover, there is some compact open subgroup $K$ such that $e_KeH=eH$, where $e_K$ is the projector $V\rightarrow V^K$. Letting $e' = e*e_K*e$, we have $e'H=eH\neq 0$.

Let $R:=e\cZ\otimes_{W(k)}A$. The $W(k)$-module $e'(\cInd\psi_A)\cong e'(\cInd\psi)\otimes_{W(k)}A$ carries the structure of an $R$-module, by considering it as an external tensor product. For convenience denote the $R$-module $e'(\cInd\psi)\otimes_{W(k)}A$ by $\fM$. 

\begin{lemma}
\label{torsionfree}
$\fM$ is finitely generated and torsion-free as an $R$-module. In particular, $\fM$ embeds in a free $R$-module.
\end{lemma}
\begin{proof}
Since $e(\cInd\psi)$ is admissible as an $e\cZ$-module \cite[Prop 5.3]{h_whitt}, $\fM$ is finitely generated as an $R$-module. 

Next, note that $e'(\cInd\psi)$ is torsion-free as an $e\cZ$-module. This follows from its torsion free-ness at characteristic zero primes. Since $A$ and $e\cZ$ are both reduced and flat over $W(k)$, the ring $R$ is reduced and flat over $W(k)$. Now, a module over a reduced ring is torsion-free if and only if it can be embedded in a free module \cite[1.5,1.7]{wiegand_modules}. Thus we focus on showing that $\fM$ can be embedded in a free $R$-module.

Since $e(\cInd\psi)$ is torsion-free over $e\cZ$ there is an embedding of $e\cZ$-modules $e'(\cInd\psi)\rightarrow (e\cZ)^r$ for some $r$. Since $W(k)\rightarrow A$ is flat, $e\cZ\rightarrow R$ is flat, since flatness is preserved under base-change. Now tensor this embedding with $R$ to get a map of $R$-modules $\fM \cong e'(\cInd\psi)\otimes_{e\cZ}R\rightarrow (e\cZ)^r\otimes_{e\cZ}R\cong R^r$, where the first isomorphism is the canonical one $$e'(\cInd\psi)\otimes_{W(k)}A\cong \Big(e'(\cInd\psi)\otimes_{e\cZ}e\cZ\Big)\otimes_{W(k)}A\cong e'(\cInd\psi)\otimes_{e\cZ}R.$$ But since flatness is preserved by base change, $A$ being flat over $W(k)$ implies $R$ flat over $e\cZ$. Hence, the map $\fM\rightarrow R^r$ is an embedding, so $\fM$ is torsion-free over $R$.
\end{proof}

\begin{lemma}
\label{Bexists}
The set $\{\fq\in\Spec(R):eH\in\fq \fM\}$ is contained in a closed subset $V$ of $\Spec(R)$ such that $V\neq \Spec(R)$. Moreover, this closed subset does not contain the generic fiber $\{\fq\in \Spec(R):\ell\notin\fq\}$.
\end{lemma}
\begin{proof}
From Lemma \ref{torsionfree}, there is an embedding $\fM\subset R^r$, so $\fq\fM\subset \fq^r$. Thus if $eH = (h_1,\dots,h_r)$ is in $\fq\fM$, each $h_i$ is in $\fq$. Hence $\fq$ is in the closed set $V:= V(h_1)\cap\dots\cap V(h_r)$. But $V\neq \Spec(R)$ because some $h_i$ is nonzero (so there is some minimal prime not containing $h_i$, by reducedness).
\end{proof}

Thus there is some nonempty open subset $D\subset \Spec(R)$ in the generic fiber consisting of points $\fq$ such that $eH\notin \fq \fM$.

\begin{lemma}
\label{infinitepointsdense}
Let $K$ be an infinite field and let $B$ be any infinite subset of $K$. Then the set of points $(X_1 - b_1,\dots,X_n - b_n)$ such that $b_i\in B$ is dense in $\Spec(K[X_1,\dots,X_n])$.
\end{lemma}
\begin{proof}
We proceed by induction on $n$. If $n=1$, we can show that every principal open subset intersects the set of points $\{(X-b)\}$. If $f\in K[X]$ were nonzero, then $f$ could not be divisible by $(X-b)$ for infinitely many $b$, whence there are points $(X-b)$ in $D(f)$.

Suppose the result holds for $n-1$. We denote by $S$ the subset of points $(X_1 - b_1,\dots,X_n - b_n)$, and choose an arbitrary $f$ nonzero in $K[X_1,...,X_n]$ and consider $V=V(f)$ the set of prime ideals containing $f$. It suffices to show that $S$ cannot be contained in $V$. Consider the map $K[X_1,...,X_n] \rightarrow K[X_1,...,X_{n-1}]$ given by $X_n \mapsto b$ for some $b\in B$. This gives the closed immersion $H\rightarrow \mathbb{A}_K^n$ of the hyperplane $H:=\{X_n = b\}$. By the induction hypothesis the subset $T$ of points $(X_1-b_1,..., X_{n-1}-b_{n-1}, X_n - b)$ is dense in $H$. Suppose $V$ contains $S$, then $V\cap H\supset S\cap H \supset T$, meaning $V\cap H = H$. Since $b$ was arbitrary we've shown that $V$ contains every one of the distinct hyperplanes $\{X_n = b\}$ for $b\in B$. In particular this means each $X_n-b$ divides $f$, which is impossible.
\end{proof}

\begin{proposition}
\label{Wpointsdense}
Let $\cK$ be $\Frac{W(k)}$ and $e$ be a primitive idempotent of $\cZ$. There is a finite extension $\cK\subset\cK'$, depending only on $e$, with rings of integers $\cO'$ such that the set of points $\fp=\ker(e\cZ\xrightarrow{f} \cO')$ for some map $f:e\cZ\rightarrow \cO'$ is dense in $\Spec(e\cZ)[\frac{1}{\ell}]$.
\end{proposition}
\begin{proof}
First, note that the proof in Lemma \ref{infinitepointsdense} carries over for polynomial rings with any number of the variables $X_i$ inverted. 

By \cite[Prop 11.1]{h_bern}, $e\cZ\otimes_{W(k)}\overline{\cK} \cong \prod_{M,\pi'}\cZ_{\overline{\cK},M,\pi'}$ where $\cZ_{\overline{\cK},M,\pi'}$ denotes the center of $\Rep_{\overline{\cK}}(G)_{M,\pi'}$. From \cite{bd} we know $\cZ_{\overline{\cK},M,\pi'}\cong (\overline{\cK}[M/M^{\circ}]^H)^{W(\pi')}$, notation being the same as in the proof of Proposition \ref{univcowhittirreducibleatminimalprimes}. Thus, there exists a complete system of primitive orthogonal idempotents $\{f_1,\dots,f_s\}$, where $s$ is the number of components in the decomposition of $e\cZ\otimes\overline{\cK}$, such that $\sum_if_i=1$ in $e\cZ\otimes_{W(k)}\overline{\cK}$, and such that $f_i(e\cZ\otimes_{W(k)}\overline{\cK}) \cong (\overline{\cK}[M/M^{\circ}]^H)^{W(\pi')}$ for some $M,\pi'$. We argue that this isomorphism is in fact defined over a finite extension of $\cK$. 

There is some finite extension $\cK_i$ of $\cK$ such that $f_i$ lives in $e\cZ\otimes \cK_i$. We now check that the construction in \cite[Prop 2.11]{bd} of the natural map $f_i(e\cZ\otimes_{W(k)}\cK_i)\rightarrow (\cK_i[M/M^{\circ}]^H)^{W(\pi')}$ of $\cK_i$-algebras can be carried out over the field $\cK_i$, and is compatible with extension of the base field. Since $M/M^{\circ}\cong \ZZ^r$ is a lattice, the ring $\cK_i[M/M^{\circ}]$ can be identified with a polynomial ring $\cK_i[X_1^{\pm1},\dots,X_{r_i}^{\pm1}]$ for some integer $r_i$ and indeterminates $X_1,\dots,X_{r_i}$. Let $\chi_X$ be the homomorphism $M/M^{\circ}\rightarrow \cK_i[X_1^{\pm1},\dots,X_{r_i}^{\pm1}]^{\times}$ which sends $(n_1,\dots,n_r)\in \ZZ^{r_i}$ to $X_1^{n_1}\cdots X_{r_i}^{n_{r_i}}$. This gives $i_P^G(\pi'\otimes\chi_X)$ the structure of a $\cK_i[X_1^{\pm1},\dots,X_{r_i}^{\pm1}]^H[G]$-module (recall that $H$ is the subgroup of unramified characters stabilizing $\pi'$). If $z$ is an element of $f_i(e\cZ\otimes_{W(k)}\cK_i)$, then $z$ defines an endomorphism in $\End_G(i_P^G(\pi'\otimes\chi_X))$, so we have described a map of rings $f_i(e\cZ\otimes_{W(k)}\cK_i)\rightarrow \End_G(i_P^G(\pi'\otimes\chi_X))$. We argue that the image $I$ of this map lives in the subring $(\cK_i[X_1^{\pm1},\dots,X_{r_i}^{\pm 1}]^H)^{W(\pi')}$ of $\End_G(i_P^G(\pi'\otimes\chi_X))$ consisting of polynomials fixed by the action of both $H$ and $W(\pi')$ ($W(\pi')$ acts by permuting the indeterminates $X_i$). Since $i_P^G(\pi'\otimes\chi)$ is irreducible for all $\chi$ in an open dense subset of $\overline{\cK}[M/M^{\circ}]^H$ (\cite[Thm 27]{b_rum}), it is absolutely irreducible for all $\chi$ in an open dense subset of $\cK_i[M/M^{\circ}]^H$. When $\fp\in\Spec(\cK_i[M/M^{\circ}]^H)$ is such that $i_P^G(\pi'\otimes\chi_X) \mod \fp$ is absolutely irreducible, $z$ acts by a scalar. Thus for an open dense subset of points $\fp$ we have $I_{\fp}\subset(\cK_i[M/M^{\circ}]^H)_{\fp}$, and since this is a closed condition, it holds for all points $\fp$. Then, since $\End_G(i_P^G(\pi'\otimes\chi_X))$ is finitely generated over $\cK_i[M/M^{\circ}]^H$, it follows that $I\subset \cK_i[M/M^{\circ}]^H$. Moreover, the action of $W(\pi')$ on the group of unramified characters translates to an action on $i_P^G(\pi'\otimes\chi_X)$ via permuting the $X_i$'s in $\cK_i[X_1^{\pm1},\dots,X_{r_i}^{\pm 1}]$. Since $z$ commutes with this action, we must have $I\subset (\cK_i[M/M^{\circ}]^H)^{W(\pi')}$.

By construction, after base changing from $\cK_i$ to $\overline{\cK}$, this map becomes the isomorphism $f_i(e\cZ\otimes_{W(k)}\overline{\cK}) \cong (\overline{\cK}[M/M^{\circ}]^H)^{W(\pi')}$ of \cite[Th\'{e}or\`{e}me 2.13]{bd}. Since $\cK_i\rightarrow \overline{\cK}$ is faithfully flat, this implies $f_ie\cZ\otimes_{W(k)}\cK_i\rightarrow (\cK_i[M/M^{\circ}]^H)^{W(\pi')}$ is an isomorphism. Let $\cK'$ be the smallest extension of $\cK$ containing all the extensions $\{\cK_1,\dots,\cK_s\}$. Since $e\cZ[\frac{1}{\ell}]\rightarrow e\cZ\otimes_{W(k)}\cK'$ is faithfully flat, the induced map on spectra is surjective. Since $e\cZ\otimes\cK'$ decomposes as a product of $f_i(e\cZ\otimes\cK')\cong(\cK'[X_1^{\pm1},\dots,X_{r_i}^{\pm1}]^H)^{W(\pi')}$, we therefore have integers $r_i$ such that there is a continuous surjection $\bigsqcup_{i=1}^s\Spec(\cK'[X_1^{\pm 1},\dots,X_{r_i}^{\pm 1}])\rightarrow \Spec(e\cZ[\frac{1}{\ell}])$. 

Let $\cO'$ be the ring of integers of $\cK'$. Lemma \ref{infinitepointsdense} tells us that the set of primes $(X_1-b_1,\dots,X_{r_i}-b_{r_i})$ for $b_j\in (\cO')^{\times}$ is dense in $\Spec(\cK'[X_1^{\pm 1},\dots,X_{r_i}^{\pm 1}])$. Therefore the set of prime ideals $\fp_{f_i}$ occurring as the kernel of a map $f_i:\cO'[X_1^{\pm 1},\dots,X_{r_i}^{\pm 1}]\rightarrow \cO'$ is dense in the generic fiber of $\Spec(\cO'[X_1^{\pm 1},\dots,X_{r_i}^{\pm 1}])$.

Thus for each $i$ the set of $\fp_{f_i}$ is dense in the generic fiber of the $i$'th component of the disjoint union. Since the image of a dense set under a surjective continuous map is dense, we get a dense set of points in the generic fiber of $\Spec(e\cZ)$, each of which is valued in $\cO'$. 
\end{proof}

Since the algebra $W(k)\rightarrow A$ is flat and finite type, the natural map $e\cZ\rightarrow R$ is flat and finite type. Let $\phi:\Spec(R)\rightarrow \Spec(e\cZ)$ be the map of spectra induced by $e\cZ\rightarrow R$. Since these rings are Noetherian, $\phi$ is open, so $\phi(D)$ forms an open subset of $\Spec(e\cZ)$. Moreover, since $D$ intersects the generic fiber of $\Spec(R)$, $\phi(D)$ intersects the generic fiber of $\Spec(e\cZ)$, and therefore contains a generic point of $\Spec(e\cZ)$. By definition, all points $\fp$ in $\phi(D)$ satisfy $eH\notin \fp\fM$.

By intersecting with the open set in Proposition \ref{univcowhittirreducibleatminimalprimes}, we get an open neighborhood of this generic point consisting of points $\fp\in\Spec(e\cZ)$ such that $eH\notin \fp\fM$ and $e(\cInd\psi)\otimes_{e\cZ}\kappa(\fp)$ is absolutely irreducible. 

Let $\cO'$ be the complete discrete valuation ring, which is finite over $W(k)$, appearing in the conclusion of Proposition \ref{Wpointsdense}. Proposition \ref{Wpointsdense} now allows us to conclude there exists an $\cO'$-valued point $f:e\cZ\rightarrow \cO'$ with $\fp:=\ker(f)\in \Spec(e\cZ)$ satisfying:
\begin{itemize}
\item[(i)] $eH\notin \fp\fM$
\item[(ii)] The fiber $e(\cInd\psi)\otimes_{e\cZ}\kappa(\fp)$ is absolutely irreducible.
\end{itemize}

We will now use this point $\fp$ to construct a Whittaker function as in Theorem \ref{rikka2}.

Define $\cO:=e\cZ/\fp\subset \cO'$. The ring $\cO$ is an $\ell$-torsion free $W(k)$-algebra which is an integral domain, occuring as an intermediate extension $W(k)\subset \cO\subset \cO'$. Since $W(k)\subset \cO'$ is a finite extension of complete DVR's, $\cO$ is a complete DVR, finite over $W(k)$. Let $g:e\cZ\rightarrow \cO$ be the surjective map given by reduction modulo $\fp$. Let $A'=\cO\otimes_{W(k)}A$ and let $g_A:R\rightarrow A'$ be the base change to $A$ of $g:e\cZ\rightarrow \cO$. Define $U\in \Rep_{\cO}(G)$, $U_A\in \Rep_{A'}(G)$ by
\begin{align*}
U:= e(\cInd\psi)\otimes_{e\cZ,g}\cO=\frac{e\cInd\psi}{\fp(e\cInd\psi)}&&
U_A:= e(\cInd\psi_A)\otimes_{R,g_A}A'=\frac{e\cInd\psi_A}{\fp(e\cInd\psi_A)}.
\end{align*}
Note that $U_A=U\otimes_{W(k)}A$.

Consider the maps $p:e(\cInd\psi)\rightarrow U$ and $p_A:e(\cInd\psi_A)\rightarrow U_A$ given by reduction modulo $\fp$. Then we have $p_A(eH)\neq 0$ by construction. Let $U_{A}^{\vee}$ be the smooth $A'$-linear dual of $U_{A}$ and $U^{\vee}$ be the smooth $\cO$-linear dual of $U$. Since $p_A(eH)\neq 0$ we can choose $v_{A}^{\vee}\in U_A^{\vee}$ such that $\langle v_{A}^{\vee},p_A(eH)\rangle\neq 0$ in $A'$.

Recall that each block of $\Rep_{W(k)}(G)$ corresponds to a primitive idempotent $e_{[L,\pi]}$, where $L$ is a standard Levi subgroup and $\pi$ is a supercuspidal $k[L]$-module, and $e_{[L,\pi]}$ projects any object $V$ onto its largest direct summand living in that block. The contragredient $\pi^{\vee}$ is also supercuspidal. We define $e_{[L,\pi]}^*:=e_{[L,\pi^{\vee}]}$.

\begin{lemma}
\label{bernsteinandduality}
Let $e$ be a primitive idempotent of $\cZ$. Then
\begin{itemize}
\item[(i)] for any $V\in \Rep_{W(k)}(G)$, $(eV)^{\vee} = e^*V^{\vee}$
\item[(ii)] given $\theta\in \cInd_N^G\psi$ and $\eta\in\Ind_N^G\psi^{-1}$, we have
$\langle e\theta,\eta\rangle = \langle \theta, e^*\eta\rangle$.
\end{itemize}
\end{lemma}
\begin{proof}
By definition, $e^*V^{\vee}$ (resp. $eV$) is the largest direct summand of $V^{\vee}$ (resp. of $V$) all of whose simple $W(k)[G]$-subquotients have mod-$\ell$ inertial supercuspidal support isomorphic to $(L,\pi^{\vee})$ (resp. $(L,\pi)$). All the simple subquotients of $(eV)^{\vee}$ occur as the duals of simple subquotients of $eV$. Thus by the duality theorem for parabolic induction, the simple subquotients of $(eV)^{\vee}$ have supercuspidal support $(L,\pi^{\vee})$. Since $(eV)^{\vee}$ a direct summand of $V^{\vee}$, and it is the largest with this property, we have $(eV)^{\vee}=e^*V^{\vee}$.

To prove the second part, recall that the pairing $\langle , \rangle$ on $\cInd\psi\times \Ind\psi^{-1}$ induces a $G$-equivariant isomorphism $\Ind\psi^{-1}\isomto (\cInd\psi)^{\vee}$, and therefore an isomorphism $e^*\Ind\psi^{-1}\isomto e^*(\cInd\psi)^{\vee} = (e\cInd\psi)^{\vee}$.
\end{proof}

We identify $e^*\Ind\psi_{A'}^{-1}$ with the $A'$-linear dual of $e(\cInd\psi_{A})$ and identify $e^*\Ind\psi_{\cO}^{-1}$ with the $\cO$-linear dual of $e(\cInd\psi)$. We formulate:
\begin{lemma}
The following diagram commutes:
$$\begin{array}[c]{ccc}
\Hom_{e\cZ[G]}(e\cInd\psi,U)\otimes_{W(k)}A&\longrightarrow&\Hom_{R[G]}(e\cInd\psi_{A},U_{A})\\
\downarrow&&\downarrow\\
\Hom_{\cO[G]}(U^{\vee},e^*\Ind\psi_{\cO}^{-1})\otimes_{W(k)}A&\longrightarrow&\Hom_{A'[G]}((U_{A})^{\vee}, e^*\Ind\psi_{A'}^{-1})
\end{array}$$
\end{lemma}
\begin{proof}
Since $U^{\vee}\otimes_{\cO}A=U_{A}^{\vee}$ and $(e\cInd\psi_{\cO})\otimes_{W(k)}A=e\cInd\psi_{A'}$, the horizontal arrows are maps of $A$-modules given by sending $\phi\otimes 1$ to the map $[h\otimes a\mapsto \phi(h)\otimes a]$. The top horizontal map is injective because $U$ is finitely generated over $e\cZ[G]$. The downward arrows are defined by $\phi \mapsto \phi_*$ where $\phi_*$ takes a map to its precomposition with $\phi$.

We now show commutativity:
$$\begin{array}[c]{ccc}
p\otimes a&\longrightarrow&[\phi:h\otimes b\mapsto p(h)\otimes ab]\\
\downarrow&&\downarrow\scriptstyle{?}\\
p_*\otimes a&\longrightarrow&[u^{\vee}\otimes b\mapsto p_*(u^{\vee})\otimes ab].
\end{array}$$
We must check that $\phi_*(u^{\vee}\otimes b)$ and $p_*(u^{\vee})\otimes ab$ are equal as elements of $(e'\cInd\psi_A)^{\vee} \cong (e'\cInd\psi)\otimes A$. But given $h\in e'\cInd\psi$ and $c$ in $A$ we have $\phi_*(u^{\vee}\otimes b)(h\otimes c) = (u^{\vee}\otimes b)(p(h)\otimes ac) = u^{\vee}(p(h))\otimes abc$.
On the other hand we have $(p_*(u^{\vee})\otimes ab)(h\otimes c) = u^{\vee}(p(h))\otimes abc$, as desired.
\end{proof}

The map $p_A \in \Hom_{R[G]}(e\cInd\psi_A,U_A)$ is in the image of the top horizontal map since it is the base change $p\otimes 1$. Thus $(p_A)_*$ equals $p_*\otimes 1$. Since $v_A^{\vee}$ is in $U^{\vee}\otimes_{\cO}A$ we can expand it as $v_A^{\vee} = \sum_iv_i^{\vee}\otimes a_i$ with $v_i^{\vee}\in U^{\vee}$ and $a_i\in A$. Then we have
\begin{align*}
0\neq \langle v_A^{\vee},p_A(eH)\rangle &= \langle (p_A)_*(v_A^{\vee}),eH\rangle\\
&=\langle (p_*\otimes 1)(v_A^{\vee}),eH\rangle\\
&=\langle (p_*\otimes 1)(\sum_iv_i^{\vee}\otimes a_i),eH\rangle\\
&=\langle \sum_ip_*(v_i^{\vee})\otimes a_i, eH\rangle\\
&=\sum_i a_i\langle p_*(v_i^{\vee}),eH\rangle
\end{align*}
This implies that not all the terms $\langle p_*(v_i^{\vee}),eH\rangle$ are zero. Therefore $\langle p_*(v_i^{\vee}),eH\rangle\neq 0$ for some $i$. Since $p_*:U^{\vee}\rightarrow e^*\Ind\psi_{\cO}^{-1}$ is $\cO[G]$-linear, it is a (the) map to the Whittaker space of $U^{\vee}$, so $p_*(v_i^{\vee})$ defines an element of $\cW(U^{\vee},\psi^{-1})$. $U$ is a co-Whittaker $\cO[G]$-module by Proposition \ref{dominance}, as it equals $e(\cInd\psi)\otimes_{e\cZ,g}\cO$. By Lemma \ref{bernsteinandduality} we conclude that $\langle p_*(v_i^{\vee}),eH\rangle = \langle e^*p_*(v_i^{\vee}),H\rangle =\langle p_*(v_i^{\vee}),H\rangle$ is nonzero.

To show that $U$ satisfies all the requirements of Theorem \ref{rikka2}, and $p_*(v_i^{\vee})$ is the required Whittaker function, the only thing left to check is that $U$ is absolutely irreducible after inverting $\ell$. If $\varpi$ is a uniformizer of $\cO$, the fact that $e(\cInd\psi)\otimes_{e\cZ}\kappa(\fp)$ is absolutely irreducible precisely means $U[\frac{1}{\varpi}]$ is absolutely irreducible, which is true by construction. The map $U\rightarrow U[\frac{1}{\varpi}]$ is an embedding because both $\cO$ and $e(\cInd\psi)$ are $\ell$-torsion free. 

Hence the $W(k)[G]$-module $U$ and the Whittaker function $p_*(v_i^{\vee})$ satisfy the conclusion of Theorem \ref{rikka2}.

\bibliography{mybibliography}{}
\bibliographystyle{alpha}
\end{document}